\documentclass[preprint, 11pt]{imsart}
\RequirePackage[T1]{fontenc}
\usepackage{color}
\usepackage[latin1]{inputenc}
\usepackage[T1]{fontenc}

\usepackage{amsmath,amssymb,graphicx}
\usepackage{epstopdf}
\usepackage{color}
\RequirePackage{amsfonts,amsthm,amsmath}
\RequirePackage[colorlinks,citecolor=blue,urlcolor=blue]{hyperref}
\usepackage{a4wide}
\usepackage{here}
\usepackage{float}
\usepackage{amsmath}
\usepackage{bbm}
\allowdisplaybreaks

\newcommand{\nbR}{\mathbb{R}}
\newcommand{\nbN}{\mathbb{N}}

\newcommand{\nbu}{\mathbbm{1}}

\newcommand{\nbP}{\mathbb{P}}

\newcommand{\nbE}{\mathbb{E}}

\newcommand{\Z}{{\cal Z}}

\renewcommand{\l}{\ell}

\newtheorem{theorem}{{\bf Theorem}}
\newtheorem{lemm}[theorem]{{\bf Lemma}}

\newtheorem{remark}[theorem]{Remark}

\def\PP{{\mathbb P}}

\newcommand{\N}{\ensuremath{\mathbb{N}}}

\newcommand{\R}{\ensuremath{\mathbb{R}}}

\begin{document}

\begin{frontmatter}

\title{Multitype branching process with nonhomogeneous Poisson and generalized Polya  immigration
}
\runtitle{Multitype branching process with immigration}

\begin{aug}
\author{\fnms{\Large{Landy}}
\snm{\Large{Rabehasaina}}
\ead[label=e2]{lrabehas@univ-fcomte.fr}} \and \
\author{\fnms{\Large{Jae-Kyung}}
\snm{\Large{Woo}}
\ead[label=e1]{j.k.woo@unsw.edu.au}}
\runauthor{L.Rabehasaina and J.-K.Woo}
%\affiliation{University Bourgogne Franche-Comt\'e}

\address{\hspace*{0cm}\\
Laboratory of Mathematics of Besan\c con,\\
University Bourgogne Franche Comt\'e,\\
16 route de Gray, 25030 Besan\c con cedex, France.\\[0.2cm]
%\hspace*{0cm} \\
\printead{e2}}

\address{\hspace*{0cm}\\
School of Risk and Actuarial Studies,\\
Australian School of Business,\\ University of New South Wales Sydney, Australia.\\[0.2cm]
%\hspace*{0cm} \\
\printead{e1}}
\end{aug}

\vspace{0.5cm}

\begin{abstract}
In a multitype branching process, it is assumed that immigrants arrive according to a nonhomogeneous Poisson or a generalized Polya process (both processes are formulated as a nonhomogeneous birth process with an appropriate choice of transition intensities). We show that the renormalized numbers of objects of the various types alive at time $t$ for supercritical, critical, and subcritical cases jointly converge in distribution under those two different arrival processes. Furthermore, some transient moment analysis when there are only two types of particles is provided.
\end{abstract}

\end{frontmatter}

\normalsize

\section{Introduction}\label{sec:intro}
We consider a multitype branching process in which there are different types of particles, and new particles arrive according to a Nonhomogeneous Poisson process (NHPP) or a generalized Polya process (GPP). Single or multitype branching processes with different stochastic assumptions on the immigration process have been applied in diverse fields in applied probability such as biology, epidemiology, and demography. For example, in \cite{R93}, the theory of multitype branching processes in discrete time with immigration was utilized to study the joint queue length process in the different queues of a polling system in queueing theory. More generally, network of infinite servers queues may be seen as multitype Galton Watson processes with immigration, see e.g. \cite{A05}, again for a discrete time model, although an extension to more general immigration arrival processes in continuous time network of infnite server queues may be available. 
%When there is no switchover time, it is related to the current model, that is, multitype branching process with immigration only in state zero. 
Some actuarial application of branching processes such as a reinsurance chain was discussed in \cite[Section 7.5]{RSST98}. Also, with regards to applications in biology, see e.g a recent paper \cite{HPYP15} which considered multitype branching processes with homogeneous Poisson immigration to study stress erythropoiesis, although the authors pointed out that an nonhomogeneous Poisson process might be more realistic in that situation. The reader is referred to \cite{MYH18} for a detailed discussion about the relevant literature on various types of branching processes. 

For the immigration processes, an alternative to homogeneous Poisson process, NHPP and GPP are chosen in this paper for the following reasons. NHPP and GPP are within the class of non-homogeneous birth processes, which means the intensity of event occurrence possibly varies with the time (e.g. seasonality of catastrophe incidence) and/or the past state of the process (e.g. number of previous shocks, the number of accidents incurred in the past). In this regard, NHPP has been widely used in various areas such as engineering, applied probability, biological science, and actuarial science. Also, the Polya process (of which marginal distribution is viewed as a gamma mixture of Poisson distribution, see e.g. \cite[Section 5]{F43}) was discussed as a good candidate for the contagion model and further, in \cite{B55}, the generalized Polya scheme was considered to take individual's accident proneness and time effect into the model. 
In the literature of risk theory, contagion model with the Polya scheme presenting a linear type of contagion was discussed to model the number of accidents by \cite[Section 2.2]{B70}; depending on the choice 
of a parameter, this model is called positive or negative contagion model.
In particular, a positive contagion model in \cite{B70} (or so-called GPP in \cite{W10,C14,CF16}) would be a suitable choice for the arrival process, which well explains contagious events in case the more event arrived in the past, the more intense of event arrivals in the future. Since a branching process can be used to study a dynamic network of the spread of infectious diseases, it is natural to consider a GPP for the immigration arrival process as a suitable choice to model the occurrence of contagious events as explained above. Also, we refer to \cite{K10} for a more detailed discussion on the GPP which was provided in the framework of a non-stationary type master equation approach in mathematical physics. Assuming organisms are damaged by shocks which arrive according to the GPP in the extreme shock model, the problem of aging of organisms was studied in \cite{CF16a}. As discussed in \cite{CF16a}, a sequence of interarrival times of shocks (which are dependent on the number of past events) is decreasing in the GPP case. So this process is a realistic choice to model escalating impacts of damage on organisms which may significantly affect mortality. %\textcolor{blue} {As the GPP models the decreasing sequence of interarrival times (regarded as
%aging or deterioration) depending on the number of the events in the past (of which their positive dependence is proved in \cite{CF16a}), it becomes a more realistic choice in modeling mortality as well as the increasing damaging impact on an aging system.}%\marge{Not sure what an organism is}
%Assuming shocks arrive according to the GPP in the extreme shock model, the problem of aging of organisms was studied in \cite{CF16a}.\marge{maybe put a bit more detail about this ageing thing?}

In this paper, our focus is to study the joint asymptotic behavior of a process representing the numbers of different types of particles alive at time $t$ when the immigration process is described by NHPP or GPP processes. Such a model may be interpreted differently in function of whether we are in an epidemic, actuarial, queueing or reliability setting. In an epidemic setting, the particles represent contaminated cells and the types represent their locations, under the assumption that those cells move to those other locations where they possibly contaminate other cells. In an actuarial setting, a particle may represent a certain type of claim or task that needs to be processed in different branches of an insurance company before being settled or in different stages of a reinsurance contract as explained in \cite[Section 7.5]{RSST98}. In a queueing setting, a particle is a customer who arrives and gets served immediately in the setting of infinite server queues and, after leaving the queue, is replicated into several new customers who are sent to other queues for the subsequent service. In a reliability setting, particles are interconnected parts in a system which can be damaged upon external shock arrivals and then are necessary to repair, or are dependent line outages in a power network which may cause cascading blackouts, see \cite{QJS16}. Besides, we consider all three different underlying branching mechanisms (supercritical, critical, and subcritical) while most papers in the literature consider the critical case, see \cite{W70, D71, W72, MYH18}. Indeed, it is well known that in the subcritical and critical cases for a continuous-time multitype Galton-Watson process, i.e. when the eigenvalue of the mean matrix of offsprings does not exceed $1$, the extinction is certain, whereas survival probability in infinite horizon is positive in the supercritical case. %In addition, the time to extinction has an exponentially decaying tail in the subcritical case while the time to extinction has infinite mean in the critical case.
%\marge{from the course note at Uchicago http://galton.uchicago.edu/~lalley/Courses/312/Branching.pdf}
%Depending on the quantity of interest in the research, one of cases would receive more interest than the other two.
These three cases definitely exhibit different behaviors of the branching process when there is immigration. For example, in the case of polling systems, the stable case corresponds to subcritical branching process and the heavy traffic limit is studied using near critical branching process in \cite{V07}. %Again,
Also, the fact that we are in critical, subcritical or supercritical condition may be adequate whether we are in one of the practical settings described above. For example, in a queueing or actuarial context it may be more plausible that we are in a critical or subcritical context, as the clients or tasks will eventually exit the system, whereas in the context of epidemiology, the rapid expansion of a particular disease in the beginning of the outbreaks may lead to consider a supercritical case. Concerning the arrival process, and as mentioned before, particular attention
in the forthcoming results is given to the case where the immigration rate increases very fast.
This is already the case in the classic GPP case, as the (stochastic) arrival rate is linear with
respect to the number of arrivals at current time, so that the interarrival time decreases with
respect to the failure rate order, as explained in
\cite{CF16a, BSC18}. As to the NHPP case, assumptions on 
the intensity function are such that an exponential behaviour for the latter is studied, yielding
different kind of asymptotic results. For a single type branching process with general lifetime distribution of the particles, the reader is referred to \cite{HMY16} and \cite{HMY17} which prove asymptotic results and functional central limit theorems respectively in the supercritical and subcritical cases.

Another aspect of this paper is that we also investigate the transient moment of the process
for two types of particles branching mechanism, when the renewal function associated with
the arrival process is explicit. See e.g.
\cite{HPYP15} for a similar study when the particular lifetime
particles have a general distribution.

%For instance, in the context of epidemiology, 
%\textcolor{red}{This is because those branching processes model the evolution of illnesses, and people want to know what happens when/if they develop rapidly and that corresponds to the supercritical case. Not 100 sure but I guess that's what interests epidemiologists. Also the mathematics may be more interesting because this is when the processes tend to infinity, and people want to find some limiting objects which are not trivial, e.g. some nice limiting distribution if one divides by some non easy renormalization. From an actuarial or management point of view, I have the feeling it would be rather the critical or subcritical case that would be more of interest, as a task eventually needs to leave the system.}
The rest of the paper is organized as follows. In Section \ref{sec:model}, multitype branching process without/with immigration and relevant assumptions are described. It is necessary to include some known results and also to introduce notation for the later analysis. In Section \ref{sec:NHPP}, NHPP is assumed for the arrival process of immigrants.
%First, the Laplace transform (LT) of the number of different types of particles, denoted in vector form as $N(t)$ is given.
Some convergence results for the distribution of the number of different types of particles, denoted in vector form as $N(t)$, are given in Theorem \ref{T1} and Theorem \ref{T1bis}, with a particular emphasis in the case when the intensity of the arrival process increases exponentially. A result in the critical case is given in Theorem \ref{T2}, which agrees with previous results in the same context in \cite{W72, V77, MYH18}. For the critical case, some remarks for homogeneous Poisson immigration and one dimensional branching process with immigration are provided in Remark \ref{RHPP}. In the following subsections \ref{sec:supercritical}, \ref{sec:subcritical}, and \ref{sec:critical} detailed proofs of Theorems \ref{T1}, \ref{T1bis} and \ref{T2} are given. Section \ref{sec:GPP} considers GPP for the immigration process. Asymptotic behaviors of $N(t)$ are studied in Theorem \ref{T3} in function of the parameters of the arrival process. The detailed proofs are included in the subsequent subsections \ref{sec:superc_gpp}, \ref{subsec:sc_gpp} and \ref{sec:crit_gpp} respectively. In the proofs of Theorems \ref{T1}, \ref{T1bis}, \ref{T2} and \ref{T3}, we shall show that, for a conveniently chosen normalizing function $g(t)$, the process $N(t)/g(t)$ converges in distribution to an identifiable limit as $t\to +\infty$ by showing that the corresponding Laplace Transform converges. Finally, some transient results for the moment when there are two types of particles in the branching process are presented in Section \ref{sec:app}. %\textcolor{green}{talk here about transient study in \cite{HPYP15} in a particular setup.}

Lastly, the following matrix notation will be used throughout the paper. For any matrix $M\in \nbR^{m\times n}$, $M'\in \nbR^{n\times m}$ will denote its transpose. $<u,v>=\sum_{i=1}^k u_i v_i$ denotes the usual inner product between two vectors $u=(u_1,...,u_k)'$ and $v=(v_1,...,v_k)'$. We will let $\mathbf{1}=(1,...,1)'$, a vector with $1$'s of appropriate dimension, $\nbR_+^k=[0,+\infty)^k$ and $\nbR_+^{*k}=(0,+\infty)^k$. Finally, we will let $L^1(\nbR_+)$ the set of integrable measurable functions from $\nbR_+$ to $\nbR$.

\section{The model}\label{sec:model}

The baseline model, a classical multitype branching process (without immigration), is described as follows. We consider a set of particles of $k$ possible types, with a type $i$ particle having exponential lifetime with mean $1/\mu_i$ for $i=1,...k$, denoted by $\mathcal{E}(\mu_i)$ for $\mu_i >0$. Upon its death, a type $i$ particle produces $Y_j^{(i)}$ copies of type $j$ particles for all $j=1,...,k$, where $(Y_{1}^{(i)},...,Y_{k}^{(i)})$ is a random vector with corresponding probabilities $p_i(\mathbf{n})=p_i(n_1,...,n_k)=\nbP (Y_j^{(i)}=n_j,\ j=1,...,k)$ for $\mathbf{n}=(n_1,...,n_k)\in \nbN^k$, and generating functions defined by
\begin{equation}\label{def_gen_fc}
h_i(z)=h_i(z_1,...,z_k)= \sum_{\mathbf{n} \in \nbN^k} p_i(\mathbf{n}) z_1^{n_1}\cdots z_k^{n_k},\quad i=1,...,k,\quad z\in [0,1]^k .
\end{equation}
In other words, $p_i(\mathbf{n})$ is the probability that type $i$ particle produces $n_1,...,n_k$ copies of type $1,...,k$ particles respectively. Then all copies evolve independently and have the same dynamics. 
Note that $p_j(0,...,0)$ is the probability that no replica is made, i.e. the probability that the particle does not produce any copies at the end of its lifetime. The mean numbers of copies from type $i$ particle are denoted by $(m_{i,1},...,m_{i,k})=(\nbE[Y_{1}^{(i)}],...,\nbE[Y_{k}^{(i)}])$.  
%Before studying the general models (a) and (b), we need first to consider 
We let the vector process $N^o(t)=(N^o_1(t),...,N^o_k(t))'$ where $N^o_j(t)$ represents the number of $j$ type particles at time $t$. In particular, at time 0 there is a single particle of type 1 (i.e. $N^o(0)=\mathbf{n}_0:=(1,0,...,0)'$). Its Laplace transform (LT) is denoted by $\varphi^o_{t}(s)=\nbE[e^{<s,N^o(t)>}]$ for $s \in(-\infty,0]^k$. According to \cite[Chapter V]{A72}, $\{N^o(t), t \geq 0\}$ is a continuous-time multitype branching process (without immigration). 
%In that case, the number of objects $N^o_j(t)$ in system \#$j$ at time $t$, is thus interpreted as the number of particles of type $j$ in this process. 

We recall some useful results which will be often used in the subsequent study. First, it is convenient to introduce a $k\times k$ matrix $A=(a_{ij})_{i,j=1,...,k}$ where the $a_{ij}$'s are defined by
\begin{equation}\label{matrixA}
a_{ij}=\mu_j(m_{ij}-\nbu_{[i=j]}),\quad i,j=1,...,k.
\end{equation}
We suppose that $A$ is regular i.e. all entries of the matrix $\exp(t_0 A)$ are positive for some $t_0>0$ (see \cite[Definition (10) p.202]{A72}). This entails that the largest eigenvalue $\rho$ of $A$ is positive and with multiplicity $1$. It is commonly known that we are in the {\it subcritical case} if $\rho<0$, in the {\it critical case} if $\rho =0$, and in the {\it supercritical case} if $\rho >0$. We let $u$ and $v$ be the $k\times 1$ right and left eigenvectors 
respectively, i.e. such that $Au=\rho u$ and $v'A=\rho v'$, with positive entries, and normalized in such a way that $<u,{\bf 1}>=1$ and $<u,v>=1$. Then, in \cite[Theorem 1 p.209]{A72} it was shown that 
$\{<u, N^o(t)e^{-\rho t}>,\ t\geq 0\}$ is a martingale. Also, from \cite[Theorem 2 p.206]{A72} the almost sure asymptotic behavior of $N^o(t)$ as $t\rightarrow +\infty$ is given in Lemma \ref{LTW} below.
\begin{lemm}\label{LTW}\normalfont
%Let $A=(a_{ij})_{i,j=1,...,k}$ be the $k\times k$ matrix defined by
%$$a_{ij}=\mu_j(m_{ij}-\nbu_{[i=j]),$$
%and let $\rho$ be the largest eigenvalue of $A$. Let $v$ be a right eigenvector associated to $\rho$, i.e. $Av=\rho v$, where we suppose that $A$ is regular [definition given in Athreya and Ney (1972)], and $v$ is paired with a left eigenvector $u$ such that $uA=\rho u$, and is such that 
%$u.{\bf 1}=1$ and $u.v=1$.
There exists a non-negative random variable (rv) $W$ such that
\[
\lim_{t\rightarrow +\infty} N^o(t) e^{-\rho t}=Wv,\qquad \mathrm{a.s.}
\]
\end{lemm}
Note that the conditional LT of $W$ given $N^o(0)=\mathbf{n}_0$
\[
\varphi_W(x):=\nbE[e^{-x W}|N^o(0)=\mathbf{n}_0],\quad x\ge 0,
\]
which will appear in the subsequent results, is in general not explicit but satisfies a particular integral equation (see \cite[Eq.(28) p.206]{A72} for detail).

We then move on to a multitype branching process with immigration which is the central stochastic process studied in this paper. Let us consider that a new particle (immigrant) arrives at time $T_i$, $i\ge  1$ and it is of type 1 (without loss of generality). Then it evolves according to the branching mechanism described at the beginning of this section. 
 %We turn back to an infinite server queueing network described as in Model (b).  
 The vector process $N(t)=(N_1(t),...,N_k(t))'$ represents the number of each type of particles at time $t$ defined as:
\begin{equation}\label{Nt}
N(t)=\sum^{S(t)}_{i=1} N^{o,i}(t-T_i),\qquad t\geq 0,
\end{equation}
where $\{N^{o,i}(t), t \geq 0\}_{i\in \N}$ are independent and identically distributed (iid) copies of $\{N^{o}(t), t \geq 0\}$ with $N^{o,i}(0)=\mathbf{n}_0$ and $\{ S(t), t\ge 0\}$ is the arrival process for new particles associated with a non-decreasing sequence $(T_i)_{i\in\nbN}$ with $T_0=0$ representing arrival times of the $i$th particle, with interarrival times $(T_i-T_{i-1})_{i\in \nbN^*}$. In other words, $N^{o,i}(t-T_i)$ is a vector of the number of particles in each system at time $t$ generated from the $i$th particle (of which type is $1$) arrived at $T_i$. Also, an underlying assumption is that $N^{o,i}_j(t)=0$ when $t<0$ for $j=1,2,...,k$. Hence, $N(t)$ is a continuous-time multitype branching process with immigration given by the process $\{ S(t), \ t\ge 0\}$. 

\section{Immigration modelled by Nonhomogeneous Poisson process (NHPP)}\label{sec:NHPP}
We assume in this section that $\{ S(t), t\ge 0\}$ is a NHPP with intensity $t\mapsto \lambda(t)>0 $, and set $\Lambda(t):= \int_0^t \lambda(y) dy$ for $t\ge 0$. %A special case is a Poisson process with intensity $\lambda>0$ which is a renewal process associated with a non-decreasing sequence $(T_i)_{i\in\nbN}$ with $T_0=0$ representing arrival times of the $i$th event, such that its interarrival times $(T_i-T_{i-1})_{i\ge 1}$ is independent and identically distributed (iid) with $T_{i-1}-T_i \sim \mathcal{E}(\lambda)$.

To study the asymptotic behavior of $N(t)$ in (\ref{Nt}) when $t\rightarrow +\infty$, we first need the LT of $N(t)$. The following result is an easy extension of \cite[Equation (2)]{D71}; see also \cite[Theorem 1]{MYH18} for a similar result that rather concerns the probability generating function of $N(t)$.

%\begin{lemm}\label{lemLTZt}\normalfont
%The LT of $N(t)$ has the following expression
%\begin{equation}\label{LTZt}
%\varphi(s)=\nbE[e^{s,N(t)}]=\exp \bigg\{\lambda \int^t_0 [\varphi_{N^o(x)} (s)-1]dx \bigg\},\qquad \forall s \in (-\infty,0]^k.
%\end{equation}
%\begin{proof}
%Since, given $S(t)=n$, $(T_1,...,T_n)$ are distributed as the ordered statistics $(U_{(1)},...,U_{(n)})$ with $(U_1,...,U_n)$ which are uniformly distributed over 0 to $t$ and independent, one gets that
%\[
%\varphi(s)=\sum^\infty_{n=0} \nbE \bigg[\exp \bigg\{<s, \sum^n_{i=1}N^{o,i} (t-U_{(i)})>\bigg\} \bigg] \times e^{-\lambda t} \frac{(\lambda t)^n}{n!}.
%\]
%Since $\sum^n_{i=1}N^{o,i} (t-U_{(i)})=\sum^n_{i=1}N^{o,i} (t-U_i)$ and by independence of $(U_1,...,U_n)$ and the process $\{N^{o,i}(t), t\geq 0\}$, one obtains
%\[
%\varphi(s)=\sum^\infty_{n=0} \bigg\{\frac{1}{t} \int^t_0 \nbE[\exp(<s, N^o(t-y)>)] dy \bigg\}^n
% \times e^{-\lambda t} \frac{(\lambda t)^n}{n!}.
%\]
%A change of variable $x:=t-y$ concludes the proof.
%\end{proof}
%\end{lemm}

\begin{lemm}\label{lemLTZt}\normalfont
The LT of $N(t)$ in (\ref{Nt}) admits the following expression
\begin{equation}\label{LTZt}
\varphi_t(s)=\nbE[e^{<s, N(t)>}]=\exp \bigg\{ \int^t_0 [\varphi^o_{t-x} (s)-1]\lambda(x)dx \bigg\}=\exp \bigg\{ \int^t_0 [\varphi^o_x (s)-1]\lambda(t-x)dx \bigg\},
\end{equation}
for all $ s \in (-\infty,0]^k$.
\begin{proof}
Since, given $S(t)=n$, $(T_1,...,T_n)$ are distributed as the ordered statistics $(U_{(1)},...,U_{(n)})$ with $(U_1,...,U_n)$ which are  independent with density $y\mapsto \frac{\lambda(y)}{\Lambda(t)} \nbu_{[0,t]}(y) $, one finds
\[
\varphi_t(s)=\sum^\infty_{n=0} \nbE \bigg[\exp \bigg\{<s, \sum^n_{i=1}N^{o,i} (t-U_{(i)})>\bigg\} \bigg] \times e^{-\Lambda (t)} \frac{(\Lambda (t))^n}{n!}.
\]
Since $\sum^n_{i=1}N^{o,i} (t-U_{(i)})=\sum^n_{i=1}N^{o,i} (t-U_i)$ and by independence of $(U_1,...,U_n)$ and the process $\{N^{o,i}(t), t\geq 0\}$, one obtains
\[
\varphi_t(s)=\sum^\infty_{n=0} \bigg\{\frac{1}{\Lambda(t)} \int^t_0 \nbE[\exp(<s, N^o(t-y)>)] \lambda(y)dy \bigg\}^n
 \times e^{-\Lambda (t)} \frac{(\Lambda (t))^n}{n!}.
\]
A change of variable $x:=t-y$ concludes the proof.
\end{proof}
\end{lemm}
The following results show that the renormalized process converges towards different limits
depending on the assumptions on the intensity of the arrival process.
%Some convergence results are given as follows:%\marge{specify what moments assumptions for  $(Y_{1}^{(i)},...,Y_{k}^{(i)})$ are required in each case}
\begin{theorem}\label{T1}\normalfont
Let us suppose that the intensity $t\mapsto \lambda(t)$ of the NHPP $\{S(t),\ t \ge 0 \}$ satisfies that $t\mapsto e^{-\rho t} \lambda(t)$ is integrable. Then
\begin{equation}\label{convergence_supercritical}
e^{-\rho t} N(t) \stackrel{\cal D}{\longrightarrow} \int_0^\infty e^{-\rho z} d{\cal Y}^W_z, \qquad t\rightarrow +\infty,
\end{equation} 
where $ \{{\cal Y}^W_t,\ t\ge 0\}$ is a nonhomogeneous compound Poisson process with intensity $y\mapsto \lambda (y)$ and jumps distributed as $Wv$.
%where $\textcolor{red}{\nu}$ is a distribution on $\nbR^k_+$ with LT given by
%\[
%\int_{\nbR^k_+} e^{<s,x>}\nu(dx)=\exp \bigg\{ \lambda \int^{+\infty}_0 [\varphi_{Wv} (se^{-\rho t})-1] dt\bigg\},\qquad  \forall s \in (-\infty,0]^k,
%\]
%where we recall that $v$ is the right eigenvector such that $Av=\rho v$.

%\item When $\rho < 0$ (subcritical case), one has the convergence in distribution as
%\begin{equation}\label{convergence_distribution_critical}
%N(t) \stackrel{\cal D}{\longrightarrow}\nu, \qquad t\rightarrow +\infty, 
%\end{equation}
% where $\nu$ is a distribution on $\nbR^k_+$ with LT given by
%\[
%\int_{\nbR^k_+} e^{<s,x>}\nu(dx)=\exp \bigg\{ \int^{+\infty}_0 [\varphi^o_y(s)-1] \lambda(y)\;  dy\bigg\},\qquad   s \in (-\infty,0]^k.
%\]
%\end{enumerate}
\begin{proof}
The proof is given in Section \ref{sec:supercritical}
\end{proof}
\end{theorem}
Although Theorem \ref{T1} is valid when the eigenvalue $\rho$ has any sign, it is especially interesting in the supercritical case $\rho>0$, as $t\mapsto e^{-\rho t} \lambda(t)\in L^1(\R_+)$ implies that the intensity $\lambda(t)$ can for example grow exponentially as $e^{\delta t}$ for $0\le \delta<\rho$. However Theorem \ref{T1} becomes less interesting in the critical case $\rho=0$ or subcritical case $\rho < 0$, as the condition $t\mapsto e^{-\rho t} \lambda(t)\in L^1(\R_+)$ roughly means that the intensity tends to $0$ potentially very fast. The following result supplements and shows  that the renormalized process $N(t)$ converges in distribution or in probability 
\begin{itemize}
\item when the intensity grows exponentially in the critical or subcritical case,
\item when the intensity grows exponentially as $e^{\delta t}$ with $\delta\ge \rho$ in the supercritical case, complementing Theorem \ref{T1}.
\end{itemize}
\begin{theorem}\label{T1bis}\normalfont
Let us suppose that the intensity $t\mapsto \lambda(t)$ of the NHPP $\{S(t),\ t \ge 0 \}$ satisfies $\lambda(t)\sim \lambda_\infty e^{\delta t}$ as $t\to +\infty$ for some $\delta\ge 0$ and $\lambda_\infty >0$. %and that we are in the subcritical case $\rho < 0$.
Then the following convergences hold as $t\to+\infty$:
\begin{eqnarray}
\mbox{(i)}\quad N(t) &  \stackrel{\cal D}{\longrightarrow}& \nu \quad \mbox{if } \rho<0\mbox{ and }\delta=0, \label{subcritical_gamma_zero}\\
\mbox{(ii)}\quad e^{-\delta t} N(t) & \stackrel{\PP}{\longrightarrow} & \lambda_\infty (\delta I -A)^{-1}\mathbf{n}_0 \quad \mbox{if } [\rho \le 0\mbox{ and }\delta>0]\mbox{ or }[\rho>0\mbox{ and } \delta>\rho], \label{subcritical_gamma_positive}\\
\mbox{(iii)}\quad e^{-\delta t} \frac{N(t)}{t} & \stackrel{\PP}{\longrightarrow} & \lambda_\infty v_1 u\ \quad \mbox{if }\rho=\delta >0 , \label{supercritical_gamma_equal}
\end{eqnarray}
where $\nu$ is a distribution on $\nbR^k_+$ with LT given by
\[
\int_{\nbR^k_+} e^{<s,x>}\nu(dx)=\exp \bigg\{ \lambda_\infty \int^{+\infty}_0 [\varphi^o_y(s)-1] \;  dy\bigg\},\qquad   s \in (-\infty,0]^k.
\]
\end{theorem}
%Finally the limit \eqref{subcritical_gamma_positive} is also valid in the critical case $\rho=0$.
\begin{proof}
The proof is presented in Section \ref{sec:subcritical}. 
\end{proof}
%\textcolor{red}{For a single type, we refer to \cite{HMY16} (supercritical case) and \cite{HMY17} (subcritical case) \marge{Give more info}}.

%Although Theorem \ref{T1} covers the critical case $\rho=0$,
In the following, we are especially interested in the particular critical case $\rho=0$. From Theorems \ref{T1} and \ref{T1bis}: if $t\mapsto\lambda(t)$ is integrable then $N(t)$ converges in distribution to $\lim_{t\to +\infty}{\cal Y}^W_t$ in \eqref{convergence_supercritical}, and if  $\lambda(t)\sim \lambda_\infty e^{\delta t}$ with $\delta>0$ then one has the convergence in probability of $e^{-\delta t}N(t)$ in \eqref{subcritical_gamma_positive}. We note that an intermediary case is worth to explore when $\lambda(t)$ does not have an explosive behaviour or, roughly speaking, does not converge to $0$. This is the case if the associated Cesaro limit $\lim_{t\to \infty}\Lambda(t)/t=\lambda_\infty $ exists, as some additional convergence result may be obtained. Before detailing this convergence result, we introduce the following quantities:
\begin{eqnarray}
 Q&:=&\frac{1}{2}\sum^k_{i,\l,n=1} \frac{\partial^2 h_i}{\partial z_\l \partial z_n}(1,...,1) u_\l u_n v_i >0,\nonumber\\
 %\label{def_Q}\\
\beta&:=& \bigg(\sum^k_{\l=1}  \mu_\l^{-1}  u_\l v_\l \bigg)\frac{u_1}{Q},\label{def_beta}\\
c&:=&\frac{ (\sum^k_{\l=1}  \mu_\l^{-1}  u_\l v_\l )^2}{Q},\label{def_c}
\end{eqnarray}
where we recall that $h_i(z)=h_i(z_1,...,z_k)$ is the generating function associated to $(p_i(\mathbf{n}))_{\mathbf{n} \in \nbN^k}$ given in \eqref{def_gen_fc}. The following result holds.%The critical case is handled as follows:
\begin{theorem}\label{T2}\normalfont
Let us assume that the moments of all orders of the random vector $(Y_{1}^{(i)},...,Y_{k}^{(i)})$ exist for all $i=1,...,k$ and the intensity admits a Cesaro finite limit $\lambda_\infty=\lim_{t\to \infty}\Lambda(t)/t > 0$. When $\rho=0$ (critical case), one has the convergence in distribution as
\begin{equation}\label{conv_distrib_critical}
\frac{N(t)}{t}\stackrel{\cal D}{\longrightarrow} {\cal Z} v\otimes  \mu^{-1} , \qquad t\rightarrow +\infty,
%\int^{+\infty}_0 e^{-t} d{\cal Y}^c_t, \qquad t\rightarrow +\infty,
\end{equation} 
where ${\cal Z}$ is a rv distributed as $\Gamma (\lambda_\infty \beta, c)$  with $v\otimes  \mu^{-1} =(v_1 \mu_1^{-1},...,v_k \mu_k^{-1})$, $\beta$ and $c$ given by \eqref{def_beta} and \eqref{def_c} respectively. Here, $\Gamma (\alpha, \theta)$ denotes the gamma distribution with a shape parameter $\alpha$ and a rate parameter $\theta$.
%where $\{{\cal Y}^c_t, t\geq 0\}$ is a compound Poisson process with intensity $\lambda \beta$ and jumps distributed as ${\cal X}\in [0,+\infty)^k$ of which distribution is given by the survival function
%\begin{equation}\label{distribution_X}
%\PP \left( {\cal X} > s\right)=\exp\left( - c \max_{i=1,...,k} \frac{s_i}{v_i \mu_i}\right),\quad s=(s_1,...,s_k)\in [0,+\infty)^k
%\end{equation}
\begin{proof}
The proof is presented in Section \ref{sec:critical}. 
\end{proof}
\end{theorem}
Thus, it turns out that, in the critical case, the support of the limits \eqref{convergence_supercritical}, \eqref{subcritical_gamma_positive} and \eqref{conv_distrib_critical} are respectively the positive half line spanned by $v$, $(\delta I -A)^{-1}\mathbf{n}_0$ and $v\otimes  \mu^{-1} $. It is worthwhile
to point out that all three results concern the critical case, but with different assumptions on
the intensity function.
\begin{remark}\normalfont\label{RHPP}
When the intensity $\lambda(t)$ is constant equal to $\lambda$, Theorem \ref{T2} is the particular case of \cite[Theorem 2]{W72} which considers general interarrival times, with a slight change of notation (in that reference $\mu_i$ stands for the mean lifetime of a type $i$ particle, as opposed to $\mu_i^{-1}$ here). See also \cite[Theorem 1]{V77} for a similar result. When $\lambda(t)$ converges to some limit $\lambda_\infty$, it converges towards the same limit in the sense of Cesaro and the limit in distribution \eqref{conv_distrib_critical} corresponds to \cite[Theorem 8]{MYH18}. The proof of Theorem \ref{T2} (given in Section \ref{sec:critical}) is however original in the following sense. Contrarily to \cite{W72} which proves the result by showing that the joint moments of $N(t)/t$ converge, it does not require renewal arguments and relevant results. Instead, we start directly with the LT \eqref{LTZt} which is expressed handily in Lemma \ref{lemLTZt} and study its convergence. Similar approach was adopted in \cite{MYH18} although the authors in \cite{MYH18} start the proof from a seemingly uniform estimate from \cite{S71} for the probability generating function of $\{N^o(t),\ t\ge 0 \}$.%Second, and as a consequence of the first point, this enables to consider non homogeneous Poisson processes, which is something not covered by \cite{W72}.
\end{remark}

%\begin{remark}[one dimensional case]\normalfont\label{Rone}
%The $k=1$ case corresponds to the classical Galton Watson process with immigration. The eigenvectors $v$ and $u$ are then scalars equal to $1$, the generating functions $h_i(s_1,...,s_k)$ amount to a single function $h(s)$. One computes easily from \eqref{def_beta} and \eqref{def_c} that $\beta=2\mu_1/h''(1)$ and $c=2\mu_1^2/h''(1)$, so that the limit in distribution on the right-hand side of \eqref{conv_distrib_critical} is ${\cal Z}\mu_1 \sim \Gamma(2\lambda(\infty) \mu_1/h''(1), 2\mu_1/h''(1))$. This limit agrees with the one dimensional case in \cite[Theorem p.1122]{D71} and \cite[Theorem 2]{W72}, so that Theorem \ref{T2} is a generalization of \cite{D71, W72} to multi-type Galton Watson processes with immigration in the particular case where the latter is modelled by a nonhomgeneous Poisson process.% In fact, Theorem \ref{T2} also revisits \cite[Theorem 2]{W72} in the Poisson case with a proof that does not involve convergence of moments but that proves directly the convergence of the corresponding Laplace transform.
%\end{remark}
\begin{remark}\normalfont
We remark that from (\ref{asyLTNtt}) in the proof of Theorem \ref{T2}, the limiting distribution of \eqref{conv_distrib_critical} admits a similar integral form as the right-hand side of \eqref{convergence_supercritical} which is available by applying Campbell's formula (the details are given in the beginning of Section \ref{sec:supercritical}). Indeed, one checks the equality in distribution of ${\cal Z}v\otimes  \mu^{-1} $ and $\int^{+\infty}_0 e^{-  t} d{\cal Y}^c_t$ where $\{{\cal Y}^c_t, t\geq 0\}$ is a compound Poisson process with intensity $\lambda_\infty \beta$ and jumps distributed as $ \chi v\otimes   \mu^{-1} $ with $\chi\sim {\cal E}(c)$.
\end{remark}

We now proceed to the proofs of Theorems \ref{T1}, \ref{T1bis} and \ref{T2}.% in the first two cases $\rho>0$ and $\rho< 0$. 
 
\subsection{Proof of Theorem \ref{T1}}\label{sec:supercritical}

We start from the LT in (\ref{LTZt}), which entails that the LT of $e^{-\rho t} N(t)$ is given by
\[
\varphi_t(se^{-\rho t})= \exp \bigg\{ \int^t_0 [\varphi^o_x (se^{-\rho t})-1] \lambda(t-x)dx \bigg\} ,\quad s\in (-\infty,0]^k.
\]
The main difficulty in the proof is to show the following convergence:
\begin{equation}\label{p1}
\int^t_0 [\varphi^o_{x} (se^{-\rho t})-1]\lambda(t-x)dx \longrightarrow \int^{+\infty}_0 [\varphi_{Wv}(s e^{-\rho x})-1] \lambda(x )\; dx,\qquad t\rightarrow +\infty,\ s \in (-\infty,0]^k,
\end{equation}
where $\varphi_{Wv}$ is the LT of $Wv$. By Campbell's formula (see \cite[Formula (2.9), Theorem 2.7 p.41]{K06}), $\exp \{ \int^{+\infty}_0 [\varphi_{Wv} (se^{-\rho y})-1]\lambda(y )  dy\}$ is the LT of $\int_0^\infty e^{-\rho z} d{\cal Y}^W_z$, where $ \{{\cal Y}^W_t,\ t\ge 0\}$ is a nonhomogeneous compound Poisson process with intensity $ y\mapsto \lambda(y)$ and jumps distributed as $Wv$. Hence one has the convergence in distribution of $e^{-\rho t} N(t)$ towards $\int_0^\infty e^{-\rho z} d{\cal Y}^W_z$ if (\ref{p1}) holds. %Let us now note that the process $\{N(t),\ t\ge 0\}$ is a particular case of the class of supercritical continuous time branching process with immigration studied in \cite{BPP18}. Since Condition (2.3) in \cite{BPP18} is satisfied (as the intensity $\lambda$ is fixed and finite here), \cite[Theorem 3.2]{BPP18} says that there is a.s. convergence of $e^{-\rho t} N(t)$ towards some random vector ${\cal Z}\in [0,+\infty)^k$, which thus happens to be distributed as $\int_0^\infty e^{-\rho z} d{\cal Y}^W_z$.
This proves \eqref{convergence_supercritical}.

So, in order to prove \eqref{p1} the main idea here is to exploit the convergence $N^o(y) e^{-\rho y} \longrightarrow Wv$ a.s. as $y \rightarrow +\infty$ given in Lemma \ref{LTW}. Studying (\ref{p1}) is equivalent to analyze the limit as $t\rightarrow +\infty$ of
\begin{equation}\label{def_Qt}
Q_t := \int^{t}_0 [\varphi^o_{t-x} (se^{-\rho t})-1] \lambda(x)\; dx = \int^{t}_0  \{ \nbE[\exp (<s, N^o(t-x)e^{-\rho t}>)]-1 \} \lambda(x) \; dx .
%&= \frac{1}{h} \int^{1}_0  \{ \nbE[\exp (<s, N^o(y/h)e^{-\rho/h}>)]-1 \} \lambda((1-y)/h)\; dy . \nonumber\\
%&=  \frac{1}{h} \int^1_0  \bigg\{ \nbE\bigg[\exp (<s, \frac{N^o(y/h)}{e^{\rho y/h}}e^{-\rho(1-y)/h}>)\bigg]-1 \bigg\} \lambda((1-y)/h)\;  dy,\nonumber
\end{equation}
That is, $Q_t$ may be expressed as
\begin{equation}\label{limith}
Q_t:=Q_{1,t}+Q_{2,t}
\end{equation}
where
\begin{eqnarray}
Q_{1,t}&:= &  \int^t_0  \bigg\{ \nbE\bigg[\exp \bigg(<s, \frac{N^o(t-x)}{e^{\rho (t-x)}}e^{-\rho x}>\bigg) - \exp(<s, Wv e^{-\rho x>})\bigg] \bigg\}\nonumber\\
 && \times \lambda(x)\; dx,\label{Q1h}\\
Q_{2,t}&:=& \int^t_0  \{ \nbE[\exp (<s, Wv e^{-\rho x>})]-1 \}  \lambda(x)\; dx. \label{Q2h}
\end{eqnarray}
%and
%\begin{equation}\label{Q2h}
%Q_{2,h}:= \frac{1}{h} \int^1_0  \{ \nbE[\exp (<s, Wv e^{-\rho(1-y)/h}>)]-1 \} \textcolor{red}{\lambda((1-y)/h)}dy.
%\end{equation}
We then separately examine the limits of (\ref{Q1h}) and (\ref{Q2h}). In the end, it will be shown that (\ref{Q1h}) tends to 0 and (\ref{Q2h}) tends to the right-hand side of \eqref{p1}.

{\bf Limit of $Q_{1,t}$ in (\ref{Q1h}) as $t\rightarrow +\infty$.} We shall utilize the following basic inequality in the subsequent proof:
\begin{equation}\label{A}
|e^a-e^b|\leq |a-b|,\qquad a\leq0,~b\leq 0,
\end{equation}
due to the finite increment formula and also we have $|e^a-e^b| \leq e^a+e^b \leq 2$. 
Hence $|e^a-e^b|\leq |a-b| \wedge 2$ for $a\leq0$ and $b\leq 0$. We then deduce that
%\[
%|Q_{1,t}| \leq   \int^t_0   \nbE\bigg[ \bigg| <s, \frac{N^o(t-x)}{e^{\rho (t-x)}}e^{-\rho x}>- <s, Wv e^{-\rho x}>\bigg| \wedge 2\bigg] \lambda(x) dx.
%\]
%In order to study the right-hand side of the above inequality, it is convenient to \textcolor{red}{make the change of variable} $t:=(1-y)/h$ first. Then, it leads to 
\begin{align}\label{Q1ha}
|Q_{1,h}| \leq &  \int^{t}_0 \nbE\bigg[\bigg| <s, \frac{N^o(t-x)}{e^{\rho (t-x)}}e^{-\rho x}>- <s, Wv e^{-\rho x}>\bigg|\wedge 2 \bigg]   \lambda(x)dx\nonumber\\
&~~~=   \int^{\infty}_0  \nbu_{[0\leq x \leq t]}  \nbE\bigg[\bigg| <s, \frac{N^o(t-x)}{e^{\rho (t-x)}}e^{-\rho x}>- <s, Wv e^{-\rho x}>\bigg|\wedge 2 \bigg]   \lambda(x)dx .
\end{align}
By the dominated convergence theorem, it will be shown that (\ref{Q1ha}) tends to zero as $t \to +\infty$ in the following. From the pointwise convergence as $t \to +\infty$ in Lemma \ref{LTW} with the help of the dominated convergence theorem, one finds that the integrand goes to zero i.e.
\[
\nbu_{[0\leq x \leq t]}    \nbE \bigg[  \bigg| e^{-\rho x} <s, \frac{N^o(t-x)}{e^{\rho (t-x)}}>-e^{-\rho x} <s, Wv> \bigg|\wedge 2 \bigg] \lambda(x)\longrightarrow 0,\quad t\rightarrow \infty ,\ \forall x\ge 0 .
\]
We next want to find an upper bound of this integrand by some function $x \mapsto f(x) \geq  0$ such that $\int^{+\infty}_0 f(x) dx  < +\infty$. Recall that $u$ is an eigenvector with positive entries $u_i$ for $i=1,...,k$ such that $Au=\rho u$ (where the elements of the matrix $A$ are defined in (\ref{matrixA})). Since $u_i >0$ for all $i$, there exists some constant $\kappa>0$ which is large enough satisfying \begin{equation}\label{sj}
0\leq -s_j \leq \kappa u_j, \qquad \forall j=1,...,k,
\end{equation}
where we recall that the vector $s=(s_1,...,s_k)$ is fixed.
For example, $\kappa$ can be chosen as $\max_{j=1,...,k} -s_j/ u_j$. Also, note that $\nbE[ (X+Y) \wedge 2] \leq (\nbE[X]+\nbE[Y]) \wedge 2$ for nonnegative random variables $X$ and $Y$. Combining these results together with the martingale property of $\{<u, N^o(t)e^{-\rho t}>, t\geq 0\}$, we conclude that the integrand is bounded as
\begin{align}\label{Q1hinteg}
&\nbu_{[0\leq x \leq t]}  \bigg\{\bigg(  e^{-\rho x} \nbE \bigg[\bigg| <s, \frac{N^o(t-x)}{e^{\rho (t-x)}}>\bigg|\bigg]+  e^{-\rho x} \nbE[ |<s, Wv>|] \bigg)\wedge 2\bigg\}  \lambda(x)\nonumber\\
&~~~= \nbu_{[0\leq x \leq t]}  
 \bigg\{\bigg( e^{-\rho x} \nbE \bigg[  <-s, \frac{N^o(t-x)}{e^{\rho (t-x)}}>\bigg]+ e^{-\rho x} \nbE[ <-s, Wv>] \bigg)\wedge 2\bigg\} \lambda(x) \nonumber\\
 &~~~\leq  \nbu_{[0\leq x \leq t]} 
  \bigg\{ \bigg(   e^{-\rho x} \kappa \  \nbE \bigg[ <u, \frac{N^o(t-x)}{e^{\rho (t-x)}}>\bigg]+   e^{-\rho x} \nbE[<-s, Wv>]\bigg)\wedge 2\bigg\} \lambda(x),
\end{align}
where the first equality is due to the fact that $s_i \leq 0$ for $i=1,2,...,k$ and $\frac{N^o(t-x)}{e^{\rho (t-x)}}$ and $Wv$ have nonnegative entries, and the last inequality is due to (\ref{sj}). The first expectation in (\ref{Q1hinteg}) is essentially $\nbE [ <u, N^o(0)/e^{\rho \times 0}>]$ because of the martingale property and in turn, it is equal to $<u, \mathbf{n}_0>=u_1$ because of $N^o(0)=\mathbf{n}_0$. And the second expectation is some finite constant. Therefore we conclude that
(\ref{Q1hinteg}) is bounded as, for some constants $K>0$ and $K^\ast>0$, 
\begin{align*}
&\nbu_{[0\leq x \leq t]} 
  \bigg\{ \bigg( e^{-\rho x} \kappa  \nbE \bigg[ <u, \frac{N^o(t-x)}{e^{\rho (t-x)}}>\bigg]+ e^{-\rho x} \nbE[<-s, Wv>] \bigg) \wedge 2\bigg\} \lambda(x)\\
  &~~~= \nbu_{[0\leq x \leq t]} [(\kappa \ u_1e^{-\rho x}+Ke^{-\rho x}) \wedge 2]\lambda(x)  \leq K^\ast e^{-\rho x} \lambda(x):=f(x).
\end{align*}
Then, it is now shown that the integrand in (\ref{Q1ha}) tends to 0 as $t\rightarrow +\infty$ for a fixed $x$ and is dominated by the function $x \mapsto f(x)$ which is integrable by assumption. Therefore, by the dominated convergence theorem we conclude that (\ref{Q1ha}) goes to 0 as $t\rightarrow \infty$, which implies that $Q_{1,t}$ in (\ref{Q1h}) verifies $\lim_{t\rightarrow \infty } Q_{1,t}=0$. 

{\bf Limit of $Q_{2,t}$ in (\ref{Q2h}) as $t\rightarrow +\infty$.} % First, the change of variable $z:=1-y)/h$ in (\ref{Q2h}) yields
%\begin{equation}\label{Q2ha}
%Q_{2,h} =   \int^{1/h}_0  \{ \nbE[\exp (<s, Wv e^{-\rho \textcolor{red}{z}}>)]-1 \}\textcolor{red}{\lambda(z)\; dz.}
%\end{equation}
In order to prove that the integral $Q_{2,t}$ converges as $t\to \infty$, it suffices to show that $x\mapsto \left| \{ \nbE[\exp (<s, Wv e^{-\rho x}>)]-1 \} \lambda(x) \right|$ is upper bounded by some integrable function. Since $<s, Wv e^{-\rho x}>\leq 0$ for $  s \in (-\infty,0]^k$ with the help of (\ref{A}), the following inequality holds:
\[
\big|\exp (<s, Wv e^{-\rho x }>)-1\big| \leq \big|<s, Wv e^{-\rho x}>\big| = e^{-\rho x} <-s, Wv>.
\]
We then arrive at the following bound
$$
\left| \{ \nbE[\exp (<s, Wv e^{-\rho x}>)]-1 \} \lambda(x)\right|\le  e^{-\rho x} \lambda(x) \nbE \left[  <-s, Wv>\right] 
$$
which indeed is integrable by the integrability assumption for $x\mapsto e^{-\rho x} \lambda(x)$.
%\begin{align}\label{Q2ha}
%Q_{2,h} &= \frac{1}{h} \int^1_0  \{ \nbE[\exp (<s, Wv e^{-\rho y /h}>)]-1 \}\textcolor{red}{\lambda(y/h)}\; dy \nonumber\\
%&:=\frac{1}{h} \left(\int^{1/2}_0 +  \int^1_{1/2} \right) \{ \nbE[\exp (<s, Wv e^{-\rho y /h}>)]-1 \} \textcolor{red}{\lambda(y/h)}\; dy.
%\end{align}
%The above two integrals are analyzed separately. First, we shall prove that the the second integral in (\ref{Q2ha}) tends to zero as $h\rightarrow 0$. Indeed, since $<s, Wv e^{-\rho s/h}>\leq 0$ for $\forall s \in (-\infty,0]^k$, the following inequality holds:
%\[
%\big|\exp (<s, Wv e^{-\rho y /h}>)-1\big| \leq \big|<s, Wv e^{-\rho y /h}>\big| = <-s, Wve^{-\rho y/h}>.
%\] 
%Hence, 
%\begin{align*}
%&\frac{1}{h} \int^{1}_{1/2} \{ \nbE[\exp (<s, Wv e^{-\rho y /h}>)]-1 \} \textcolor{red}{\lambda(y/h)}\; dy \leq
%\frac{1}{h} \int^{1}_{1/2}  \nbE[<-s, Wve^{-\rho y/h}>] \textcolor{red}{\lambda(y/h)}\; dy\\
%&~~~\textcolor{red}{\le C }\nbE[<-s, Wv>] \left\{\frac{1}{h}\int^{1}_{1/2}e^{-\rho y/h}\; dy \right\}
%= \textcolor{red}{C} \nbE[<-s, Wv>] \rho (e^{-\frac{\rho}{2h}}-e^{-\frac{\rho}{h}})  \longrightarrow 0
%\end{align*}
%as $h\rightarrow 0$, \textcolor{red}{where we recall that $t\mapsto\lambda(t)$ is upper bounded by some constant $C$}. Next, let us look at the first integral in (\ref{Q2ha}). Reverting to the variable $t:=y/h$ yields that it becomes
%\[
%\int^{\frac{1}{2h}}_0   \{ \nbE[\exp (<s, Wv e^{-\rho t}>)]-1 \}\textcolor{red}{\lambda(t)}\; dt \xrightarrow[h\rightarrow 0]{}\int^{+\infty}_0   \{\varphi_{Wv} (se^{-\rho t})-1 \}dt.
%\] 
Combining the above results, the limit of $Q_{2,t}$ in (\ref{Q2h}) is obtained as
\[
Q_{2,t}\longrightarrow \int^{+\infty}_0   \{\varphi_{Wv} (se^{-\rho x})-1 \}\lambda(x)\;dx,\qquad t\rightarrow \infty .
\]
We conclude thus that the limit of $Q_t $ in (\ref{limith}) is given by the right-hand side in the above limit,
%\[
%Q_h   \longrightarrow \int^{+\infty}_0   \{\varphi_{Wv} (se^{-\rho t})-1 \}\textcolor{red}{\lambda(t)}\;dt,\qquad h\rightarrow 0
%\]
and in turn, that (\ref{p1}) is proved. Consequently, this completes the proof.

\subsection{Proof of Theorem \ref{T1bis}}\label{sec:subcritical}
{\it Proving (i) the limit of $N(t)$ in \eqref{subcritical_gamma_zero} as $t\to +\infty$.} First, we recall that $\rho<0$ and $\delta=0$, and the
intensity satisfies $\lim_{t\infty}\lambda(t)=\lambda_\infty$. We proceeed to prove straightforwardly that $\int^{t}_0 [\varphi^o_x (s)-1]\lambda(t-x)  dx$ in Lemma \ref{lemLTZt} converges to $\lambda_\infty\int^{\infty}_0 [\varphi^o_x (s)-1]  dx$ as $t\to +\infty$ by a dominated convergence argument, so that the LT of $N(t)$ in \eqref{LTZt} converges to  $\exp \big\{ \lambda_\infty\int^{+\infty}_0 [\varphi^o_x(s)-1]  dx\big\}$, which from e.g. \cite[Theorem 2]{K51} is the LT of some distribution $\nu$ with support in $\nbR^k_+$. We start by writing
\begin{equation}\label{gamma_zero}
\int^{t}_0 [\varphi^o_x (s)-1]\lambda(t-x)  dx = \int^{\infty}_0 \nbu_{[0\le x\le t]}[\varphi^o_x (s)-1]\lambda(t-x)  dx .
\end{equation}
Using the inequality in (\ref{sj}), one finds
\[
|<s,N^o(x)>|=<-s,N^o(x)> \leq  \kappa <u, N^o(x)>,
\]
where we recall that $\kappa=\max_{j=1,...,k} -s_j/u_j$ for example. Then we get that
\begin{align}
&\nbu_{[0\le x\le t]} |\varphi^o_x(s)-1|\lambda(t-x)= \nbu_{[0\le x\le t]}|\nbE[e^{<s, N^o(x)>}]-1| \lambda(t-x) \nonumber\\
&~~~ \leq \nbu_{[0\le x\le t]}\nbE[ |e^{<s,N^o(x)>}-1|] \lambda(t-x) \leq \nbu_{[0\le x\le t]} \nbE[ |<s,N^o(x)>|] \lambda(t-x)\nonumber\\
& ~~~ \leq  C_\lambda \kappa\ \nbu_{[0\le x\le t]}\nbE[<u, N^o(x)>] = C_\lambda \kappa \nbu_{[0\le x\le t]} e^{\rho x} \nbE[<u, N^o(x) e^{-\rho x}>] \nonumber\\
&~~~= C_\lambda \kappa \nbu_{[0\le x\le t]} e^{\rho x}  \nbE[<u, N^o(0)>]=C_\lambda \kappa e^{\rho x} \nbE[<u,\mathbf{n}_0>]= C_\lambda \kappa u_1e^{\rho x},\label{ineq_critical}
\end{align} 
where $C_\lambda =\sup_{y\ge 0} \lambda(y)$ is a finite quantity since $\lambda(y)$ converges as $y\to +\infty$. Since $\rho<0$, the right-hand side of \eqref{ineq_critical} is integrable, and $\lim_{t\to +\infty} \nbu_{[0\le x\le t]}[\varphi^o_x (s)-1]\lambda(t-x)=[\varphi^o_x (s)-1]\lambda_\infty$, the dominated convergence theorem entails that \eqref{gamma_zero} indeed converges to $\lambda_\infty\int^{\infty}_0 [\varphi^o_x (s)-1]  dx$ as $t\to +\infty$.

{\it Proving (ii) the limit of $e^{-\delta t}N(t)$ in \eqref{subcritical_gamma_positive} as $t\to +\infty$.} We assume here that $\rho\le 0$ and $\delta>0$ or that $\rho> 0$ and $\delta>\rho$. First, the LT of  $e^{-\delta t}N(t)$ is given as $e^{R_t}$ with $R_t$ which can be written as $R_{1,t} + R_{2,t}$ as below:
\begin{eqnarray}
R_t&=& \int^{t}_0 [\varphi^o_{t-x} (se^{-\delta t})-1] \lambda(x)\; dx = \int^{t}_0  \{ \nbE[\exp (<se^{-\delta t}, N^o(t-x)>)]-1 \} \lambda(x) \; dx \label{dec_Qt_delta}\\
&=&R_{1,t}+R_{2,t}, \nonumber\\
R_{1,t}&:=& \int_0^\infty \nbu_{[0\le x\le t]}\nbE\left[ \exp\left( <se^{-\delta t}, N^o(t-x)>\right)- 1 - <se^{-\delta t}, N^o(t-x)>\right]\lambda(x)\; dx,\label{def_R1t}\\
R_{2,t}&:=& \int_0^\infty \nbu_{[0\le x\le t]}\nbE\left[ <se^{-\delta t}, N^o(t-x)>\right]\lambda(x)\; dx .\label{def_R2t}
\end{eqnarray}
Similar to \eqref{def_Qt}, we then study the limits of $R_{1,t}$ and $R_{2,t}$ separately as $t\to +\infty$.\\
{\bf Limit of $R_{1,t}$ in \eqref{def_R1t} as $t\to +\infty$. }It will be shown that $R_{1,t}$ tends to $0$ as $t\to +\infty$. The finite increment formula applied to $\psi(x):=e^x-1-x$ as well as the inequality $|e^u-1|\le |u|$ for $u\le 0$ implies that, for all $x\le 0$,
%\begin{multline}\label{ineg_psi}
$$|\psi(x)-\psi(0)|=|e^x-1-x|\le \sup_{u\in [x,0]} |\psi'(u)| . |x| = \sup_{u\in [x,0]} |e^u-1| . |x|\le \sup_{u\in [x,0]} |u| . |x|=|x|^2.$$
%\le \sup_{u\in [x,0]} (|u|\wedge 2) . |x|=(|x|\wedge 2)|x|= |x|^2 \wedge (2 |x|).
%\end{multline}
%Here, we prove that $\lim_{t\to +\infty}R_{1,t}=0$ by considering both cases $\delta >0 $ and $\delta\in (\rho,0)$ seperately, and we start by the case $\delta>0$.
%The inequality \eqref{ineg_psi} implies $|e^x-1-x|\le |x|^2$ for all $x\le 0$ which, combined with the Cauchy Schwarz inequality, yields in that case the upper bound
The above result, combined with the Cauchy-Schwarz inequality, yields the upper bound for $|R_{1,t}|$ given by
\begin{multline}\label{ineg_R1t}
|R_{1,t}| \le \int_0^\infty \nbu_{[0\le x\le t]}\nbE\left[ |<se^{-\delta t}, N^o(t-x)>|^2\right]\lambda(x)\; dx\\
\le ||s||^2\int_0^\infty \nbu_{[0\le x\le t]}e^{-2\delta t}\nbE\left[ ||N^o(t-x)||^2\right]\lambda(x)\; dx.
\end{multline}
We here separate the cases $\rho<0$, $\rho=0$ and $\rho>0$, the last case requires the additional constraint $\delta>\rho$. If $\rho<0$, \cite[Limit (19) p.204]{A72} implies that $\nbE\left[ ||N^o(t-x)||^2\right]\le C e^{\rho(t-x)}$ for some constant $C>0$. Also, the assumption $\lambda(x)\sim_{x\to\infty} \lambda_\infty e^{\delta x}$ in particular implies that $\lambda(x)$ is bounded by $e^{\delta x}$ up to a constant, hence one gets from \eqref{ineg_R1t} that for some (different) constant $C>0$,
\begin{multline*}
|R_{1,t}| \le C \int_0^\infty \nbu_{[0\le x\le t]}e^{-2\delta t} e^{\rho(t-x)}e^{\delta x} dx = C e^{(-2\delta+\rho)t} \int_0^t e^{(-\rho+\delta)x}dx\\
= \frac{C}{-\rho + \delta}[e^{-\delta t}- e^{(-2\delta+\rho)t}]\longrightarrow 0\quad \mbox{as } t\to +\infty .
\end{multline*}
If $\rho=0$ then \cite[Limit (20) p.204]{A72} implies that $\nbE\left[ ||N^o(t-x)||^2\right]$ is less than $t-x$ up to a constant, hence for some constant $C>0$ we have
\begin{multline*}
|R_{1,t}| \le C \int_0^\infty \nbu_{[0\le x\le t]}e^{-2\delta t} (t-x) e^{\delta x} dx\le t \int_0^t e^{-2\delta t} e^{\delta x} dx \\
= \frac{t}{\delta} [e^{-\delta t}-e^{-2\delta t}]\longrightarrow 0\quad \mbox{as } t\to +\infty .
\end{multline*}
Finally, if $\rho>0$ then \cite[Limit (21) p.204]{A72} implies that $\nbE\left[ ||N^o(t-x)||^2\right]$ is less than $e^{2\rho (t-x)}$, hence for some constant $C>0$ we have
\begin{equation}\label{R1t_rho_positive}
|R_{1,t}| \le C \int_0^\infty \nbu_{[0\le x\le t]}e^{-2\delta t} e^{2\rho(t-x)}e^{\delta x} dx = C e^{2(-\delta+\rho)t} \int_0^t e^{(-2\rho+\delta)x}dx
%= \frac{C}{-\rho + \delta}[e^{-\delta t}- e^{(-2\delta+\rho)t}]\longrightarrow 0\quad \mbox{as } t\to +\infty .
\end{equation}
which one can show tends to $0$ as $t\to +\infty$, thanks to $\delta>\rho$.\\
%We now consider the case $\delta\in (\rho,0)$. The inequality \eqref{ineg_psi} implies $|e^x-1-x|\le 2|x|$ for all $x\le 0$, so that the following upper bound holds:
%\begin{equation}\label{ineg_R1t2}
%|R_{1,t}| \le \int_0^\infty \nbu_{[0\le x\le t]}\nbE\left[ |<2se^{-\delta t}, N^o(t-x)>|\wedge 1\right]\lambda(x)\; dx
%\end{equation}
{\bf Limit of $R_{2,t}$ in \eqref{def_R2t} as $t\to+\infty$, and conclusion.} From \cite[p.202]{A72}, we know that the mean matrix of the multitype process $N^o(z)$ is expressed as $\nbE[N^o(z)]=e^{Az}\mathbf{n}_0$ where the matrix $A$ is defined in (\ref{matrixA}) and $\mathbf{n}_0=(1,0,...,0)'$. Therefore, $R_{2,t}$ can be expressed, after some manipulation, as
\begin{eqnarray}
R_{2,t} & = & \int_0^t e^{-\delta t}<s, e^{A(t-x)}\mathbf{n}_0> \lambda(x) dx \nonumber\\
&=& \int_0^\infty \nbu_{[0\le x\le t]} e^{-\delta t}<s, e^{Ax}\mathbf{n}_0> \lambda(t-x) dx \label{expression_R2t}
\end{eqnarray}
We now wish to use the dominated convergence theorem in order to find the limit in \eqref{expression_R2t}. Upper bounding $\lambda(x)$ by $C e^{\delta x}$ for some constant $C>0$ results in
$$
 \nbu_{[0\le x\le t]} e^{-\delta t}<s, e^{Ax}\mathbf{n}_0> \lambda(t-x) \le C<s, e^{(A-\delta I)x}\mathbf{n}_0>
$$
which is integrable because $\delta>\rho$ (in either case $\rho\le 0$ or $\rho>0$) so that all eigenvalues of $A-\delta I$ have negative real parts. Also, the assumption that $\lambda(x)\sim_{x\to\infty} \lambda_\infty e^{\delta x}$ results in $\nbu_{[0\le x\le t]} e^{-\delta t}<s, e^{Ax}\mathbf{n}_0> \lambda(t-x)\longrightarrow \lambda_\infty <s, e^{(A-\delta I)x}\mathbf{n}_0>$ as $t\to +\infty$, for all $x\ge 0$. Hence we deduce from \eqref{expression_R2t} that
$$
R_{2,t}\longrightarrow \lambda_\infty \int_0^\infty <s, e^{(A-\delta I)x}\; \mathbf{n}_0> dx= <s, \lambda_\infty \int_0^\infty e^{(A-\delta I)x} dx \; \mathbf{n}_0>=<s, \lambda_\infty (\delta I-A)^{-1}\mathbf{n}_0>.
$$
%so that, since $R_t = R_{1,t}+ R_{2,t}\longrightarrow 0+ <s, \lambda_\infty (\delta I-A)^{-1}\mathbf{n}_0>$ as $t\to +\infty$, implies that the limiting
Since $R_{1,t}\longrightarrow 0$ as $t\to +\infty$, one arrives at the convergence of the LT of $e^{-\delta t}N(t)$ to $\exp(<s, \lambda_\infty (\delta I-A)^{-1}\mathbf{n}_0>)$, so that $e^{-\delta t}N(t)$ converges in distribution (or, equivalently, in probability) towards $\lambda_\infty (\delta I-A)^{-1}\mathbf{n}_0$. Hence (ii) in \eqref{subcritical_gamma_positive} is proved.

{\it Proving (iii) the limit of $e^{-\delta t} \frac{N(t)}{t}$ in \eqref{supercritical_gamma_equal} as $t\to +\infty$.} We assume the supercritical case $\rho>0$ and $\rho=\delta$. In this case, instead of \eqref{dec_Qt_delta}, we consider the quantity $R_t:=\int^{t}_0 [\varphi^o_{t-x} (se^{-\delta t}/t)-1] \lambda(x)\; dx=\int^{t}_0  \{ \nbE[\exp (<se^{-\delta t}/t, N^o(t-x)>)]-1 \} \lambda(x) \; dx$ such that the LT of $e^{-\delta t} \frac{N(t)}{t}$  is equal to $e^{R_t}$, which is similarly decomposed as in \eqref{def_R1t} and \eqref{def_R2t} as $R_t=R_{1,t}+R_{2,t}$
\begin{eqnarray}
R_{1,t}&:=& \int_0^\infty \nbu_{[0\le x\le t]}\nbE\left[ \exp\left( <se^{-\delta t}/t, N^o(t-x)>\right)- 1 - <se^{-\delta t}/t, N^o(t-x)>\right]\lambda(x)\; dx,\label{def_R1t_equal}\\
R_{2,t}&:=& \int_0^\infty \nbu_{[0\le x\le t]}\nbE\left[ <se^{-\delta t}/t, N^o(t-x)>\right]\lambda(x)\; dx .\label{def_R2t_equal}
\end{eqnarray}
Utilizing the inequality in \eqref{R1t_rho_positive}, one obtains, for some constant $C>0$,
$$
|R_{1,t}| \le \frac{1}{t^2} C \int_0^\infty \nbu_{[0\le x\le t]}e^{-2\delta t} e^{2\rho(t-x)}e^{\delta x} dx = C  \frac{1}{t^2}  \int_0^t e^{-\delta x}dx\longrightarrow 0,
$$
thus \eqref{def_R1t_equal} tends to $0$ as $t\to +\infty$. Turning to \eqref{def_R2t_equal}, we write, using the expression $\nbE[N^o(z)]=e^{Az}\mathbf{n}_0$ as in \eqref{expression_R2t} and perform the change of variable $z:=1-x/t$ together with the assumption $\delta=\rho$:
\begin{eqnarray}
R_{2,t} & = & \frac{1}{t}\int_0^t e^{-\delta t}<s, e^{A(t-x)}\mathbf{n}_0> \lambda(x) dx \nonumber\\
&=& <s, \int_0^1 e^{A t(1-z)}e^{-\rho t} \lambda(tz) dz \; \mathbf{n}_0>\nonumber\\
&=& <s, \int_0^1 e^{(A- \rho I) t(1-z)}e^{-\rho t z} \lambda(tz) dz \; \mathbf{n}_0>.\label{compute_R2t_equal}
\end{eqnarray}
We now wish to investigate \eqref{compute_R2t_equal} when $t\to +\infty$. Since $A$ is regular, the Perron Frobenius theory entails that $e^{(A-\rho I)x}$ converges to $uv'$ as $x\to +\infty$, see e.g. \cite[Limit (17) p.203]{A72}. Hence, for all $z\in (0,1)$ one has $\lim_{t\to \infty}e^{(A-\rho I) t(1-z)}=uv'$. Also, the assumption $\lambda(x)\sim_{x\to\infty} \lambda_\infty e^{\delta x}$ with $\delta=\rho$ implies that $\lim_{t\to \infty} e^{-\rho t z} \lambda(tz) =\lambda_\infty$ for all $z\in (0,1)$, so that by the dominated convergence theorem we may let $t\to +\infty$ in \eqref{compute_R2t_equal} and obtain
$$
R_{2,t}\longrightarrow <s, \lambda_\infty uv'\; \mathbf{n}_0>=<s, \lambda_\infty v_1 u>,\quad t\to +\infty.
$$
Therefore, since $R_t=R_{1,t}+R_{2,t}$ with $\lim_{t\to \infty}R_{1,t}=0$, one concludes the convergence \eqref{supercritical_gamma_equal}.

 \subsection{Proof of Theorem \ref{T2} in the critical case $\rho=0$}\label{sec:critical}

%In this case, one has by the dominated convergence theorem $\varphi^o_{t}(s)=\nbE[\exp(<s,N^o(t)>)]$ tends to $\varphi_{Wv}(s)=\nbE[\exp(<s,Wv>)]$ as $t\rightarrow +\infty$. If $\nbP(W>0)=1$ then $\varphi_{Wv}(s)<1$ when $s \in (-\infty,0)^k$  and the integral $\int^{+\infty}_0 [\varphi^o_y (s)-1] \lambda(y) dy$ does not converge. %In fact, one has the following refinement (which is also valid when $W=0$ a.s.):
%\begin{theorem}\label{T2}\normalfont
%In the critical case $\rho=0$ one has the convergence in this case where the replication number admits moments of all order:
%\[
%\frac{Z(t)}{t} \xrightarrow[]{\mathcal{D}} \mu_c \sim \int^{+\infty}_0 e^{-t} dy^c_t, \qquad t\rightarrow +\infty,
%\]
%for some compound Poisson process $\{y^c_t, t\geq 0\}$.
%\end{theorem}
Again, we begin from Lemma \ref{lemLTZt}, from which we deduce that the LT of $N(t)/t$ admits the expression
\begin{align}\label{LTNtt}
\nbE\bigg[\exp \bigg(<s, \frac{N(t)}{t}>\bigg)\bigg] = \nbE[\exp (<t^{-1}s, N(t)>)]= \exp \bigg\{ \int^t_0  [ \varphi^o_{t-y}(t^{-1}s)-1]\lambda(y) dy \bigg\}.
\end{align}
We thus study
\begin{align}\label{criti}
\int^t_0  [ \varphi_{t-y} (t^{-1}s )-1 ] \lambda(y) dy 
&= \int^t_0   \nbE\bigg[\exp \bigg(<s, \frac{N^o(t-y)}{t}>\bigg)-1\bigg] \lambda(y) dy \nonumber\\
&= \int^{\Lambda(t)}_0   \nbE\bigg[\exp \bigg(<s, \frac{N^o(t-\Lambda^{-1}(y))}{t}>\bigg)-1\bigg]  dy \nonumber\\
&=  \int^1_0 \Lambda(t)\; \nbE\bigg[\exp \bigg(<s, \frac{N^o( t - \Lambda^{-1} (\Lambda(t) x))}{t}>\bigg)-1\bigg] dx\nonumber\\
&:=-\int^1_0 \gamma_t(x) dx,
\end{align}
where $\Lambda^{-1}(.)$ is the inverse of the function $\Lambda(.)$ (invertible as it is assumed that $\lambda(t)>0$ for all $t\geq 0$), the second last equality is due to a change of variable with $x:=y/\Lambda(t)$ and $\gamma_t(x)$ is given by
\begin{equation}\label{gammat}
\gamma_t(x):= \Lambda(t)\; \nbE\bigg[1-\exp \bigg(<s, \frac{N^o( t - \Lambda^{-1} (\Lambda(t) x))}{t}>\bigg)\bigg].
\end{equation}
We note that the assumption $\lim_{t\to +\infty}\Lambda(t)/t=\lambda_\infty$ implies that $\lim_{t\to +\infty}\Lambda^{-1}(t)/t=\lambda_\infty^{-1}$, which is in turn equivalent to 
\begin{equation}\label{Laminv}
\Lambda^{-1}(t)\sim \lambda_\infty^{-1} t, \qquad \mathrm{i.e.}\qquad \Lambda^{-1}(t)=\lambda_\infty^{-1} t + \eta(t) t,
\end{equation}
where $\lim_{t\to \infty}\eta(t)=0$.

In the following, we shall prove by the dominated convergence theorem that the right-hand side of (\ref{criti}) has the following limit
\begin{equation}\label{limit_to_prove}
-\int^1_0 \gamma_t(x) dx \longrightarrow \lambda_\infty \beta \int_0^\infty \nbE [\exp (<s, e^{-y} {\cal X}>)-1 ]dy, \quad t\to +\infty,
\end{equation}
where $\beta$ is given by \eqref{def_beta}. Here ${\cal X}=\chi v\otimes  \mu^{-1} \in [0,+\infty)^k$ where $\chi \sim {\cal E}(c) $ for $c>0$ given by \eqref{def_c} and the survival function of ${\cal X}$ is given by 
\begin{equation}\label{distribution_X}
\PP \left( {\cal X} > z\right)=\exp\left( - c \max_{i=1,...,k} \frac{z_i}{v_i  \mu_i^{-1} }\right),\quad z=(z_1,...,z_k)\in [0,+\infty)^k .
\end{equation}
 The proof is decomposed in the following steps.

{\bf Step 1: Dominating the integrand in \eqref{criti}.} First, since $s=(s_1,...,s_k) \in (-\infty, 0]^k$ has negative entries, we have for all $t\geq 0$ and $x\in(0,1)$ that
\[
0 \leq 1-\exp \bigg(<s, \frac{N^o( t - \Lambda^{-1} (\Lambda(t) x))}{t}>\bigg) \leq \, <-s, \frac{N^o( t - \Lambda^{-1} (\Lambda(t) x))}{t}>,
\]
where the last inequality is due to the fact that $1- e^{-x} \leq x$ for all $x\geq 0$. Using (\ref{sj}) again, one finds $<-s, \frac{N^o(t - \Lambda^{-1} (\Lambda(t) x) )}{t}> \leq \kappa \big(<u, \frac{N^o(t - \Lambda^{-1} (\Lambda(t) x) )}{t}>\big)$. Hence, taking the expectation and multiplying by $\Lambda(t)$ on both sides result in 
\begin{align}
0 & \leq  \gamma_t(x) \leq \Lambda(t) \kappa\ \nbE \bigg[  <u,  \frac{N^o( t - \Lambda^{-1} (\Lambda(t) x))}{t}  >\bigg]
 \le C_\lambda \kappa \ \nbE [<u, N^o(t - \Lambda^{-1} (\Lambda(t) x))>] \nonumber\\
 =& C_\lambda \kappa\ \nbE[<u,N^o(0)>]=C_\lambda\kappa\ \nbE[<u,\mathbf{n}_0>]=C_\lambda\kappa u_1, \label{domination_gamma_t_x}
\end{align}
where the first equality is obtained by the martingale argument and $C_\lambda:=\sup_{t\ge 0}\Lambda(t)/t <+\infty$. Since $C_\lambda$ and $\kappa$ are constants (independent of $t$ and $x$), the integrand in (\ref{criti}) is dominated by some constant independent from $t\ge 0$ and $x\in (0,1)$. 

{\bf Step 2: Almost sure limit of the integrand in \eqref{criti}.} Second, let us now prove the following convergence for (\ref{gammat}):
\begin{equation}\label{critiint} 
\gamma_t(x) \longrightarrow  \gamma(x):= \lambda_\infty \frac{\beta}{1-x} \nbE [1-\exp  (<s, (1-x)\;  {\cal X}> )  ],
\end{equation}
for a fixed $x\in (0,1)$ and $s=(s_1,...,s_k)\in (-\infty ,0)^k$ as $t\to +\infty$. First, using that for all $s_j <0$, $-e^{s_jx}=\int^\infty_x s_je^{s_j y}dy$, one finds $-\exp(s_j N^o_j(t - \Lambda^{-1} (\Lambda(t) x))/t)=\int_{\nbR_+^*} s_j \exp(s_j z_j)\nbu_{[z_j \ge N^o_j(t - \Lambda^{-1} (\Lambda(t) x))/t]} dz_j$ for $j=1,...,k$. Together with Fubini's theorem, we get that
\begin{equation}\label{critical_demo1}
\exp \bigg(<s, \frac{N^o(t - \Lambda^{-1} (\Lambda(t) x))}{t}>\bigg)=(-1)^k\int_{\nbR_+^{*k}} \prod_{j=1}^k \left[s_j \exp(s_j z_j)\nbu_{[z_j \ge N^o_j(t - \Lambda^{-1} (\Lambda(t) x))/t]}\right] dz,
\end{equation}
where $dz=dz_1 \cdots dz_k$, $z=(z_1,...,z_k)\in\nbR_+^{*k}$. 
Since it is necessary to have the integral $\int_{\nbR_+^{*k}} \prod_{j=1}^k [s_j \exp(s_j z_j)] dz$ convergent later, we consider the case when $s_j<0$ for all $j=1,...,k$. However, it is not hard to check that the proof can be also accommodated the case when one of the $s_j$'s is zero.
%where 
%\[
%\gamma(y):= \sum_{J\subset\{1,...,k\}}(-1)^{|J|} \prod^k_{i=1}s_i \int_{\nbR^k_{+}}
%\frac{\beta}{x}\exp \bigg(-\frac{c}{x} \max_{i \in J} \frac{y_i}{\mu_i v_i} \bigg)  e^{<s,\underline{y}>}  d\underline{y}
%\]
%with 
%\[
%\beta:= \bigg(\sum^k_{\l=1} \mu_\l u_\l v_\l \bigg)\frac{u_1}{Q(u)},\qquad 
%c:=\frac{ (\sum^k_{\l=1} \mu_\l u_\l v_\l )^2}{Q(u)}.
%\]
%Here, 
%\[
%Q(u):=\frac{1}{2}\sum^k_{i,\l,n=1} \frac{\partial^2 h_i}{\partial s_\l \partial s_n}(1,...,1) u_\l u_n v_i >0
%\] 
The results in \cite{W70} will be repeatedly used in the following leading to the convergence (\ref{critiint}). 
%Utilizing the result of $\exp(s_j N^o_j(t - \Lambda^{-1} (\Lambda(t) x))/t)=\int_{\nbR_+^*} s_j \exp(s_j z_j)\nbu_{[z_j \ge N^o_j(t - \Lambda^{-1} (\Lambda(t) x))/t]} dz_j$ for $j=1,...,k$ together with Fubini's theorem, we get that
%\begin{equation}\label{critical_demo1}
%\exp \bigg(<s, \frac{N^o(t - \Lambda^{-1} (\Lambda(t) x))}{t}>\bigg)=\int_{\nbR_+^{*k}} \prod_{j=1}^k \left[s_j \exp(s_j z_j)\nbu_{[z_j \ge N^o_j(t - \Lambda^{-1} (\Lambda(t) x))/t]}\right] dz,
%\end{equation}
%where \textcolor{red}{$dz=dz_1 \cdots dz_k$}, $z=(z_1,...,z_k)\in\nbR_+^{*k}$. 
By an expansion formula, one has
\begin{multline*}
\prod_{j=1}^k \nbu_{[z_j \ge N^o_j(t - \Lambda^{-1} (\Lambda(t) x))/t]}= \prod_{j=1}^k \left[1- \nbu_{[z_j < N^o_j(t - \Lambda^{-1} (\Lambda(t) x))/t]}\right]
\\ = 1+ \sum_{J\subset \{1,...,k\}} (-1)^{\mbox{card}(J)}\prod_{j\in J} \nbu_{[z_j < N^o_j(t - \Lambda^{-1} (\Lambda(t) x))/t]},
\end{multline*}
where $\sum_{J\subset \{1,...,k\}}$ is the sum over nonempty sets $J\subset \{1,...,k\}$. Plugging the above expression into \eqref{critical_demo1}, it follows that $\gamma_t(x)$ in \eqref{gammat} may be expressed as 
\begin{eqnarray}
\gamma_t(x)&= &\sum_{J\subset \{1,...,k\}} (-1)^{\mbox{card}(J)+k+1} \nonumber\\
&& \times \int_{\nbR_+^{*k}} \prod_{j=1}^k [s_j \exp(s_j z_j)] \; \Lambda(t)\; \nbP\bigg(\frac{N^o_j(   t - \Lambda^{-1} (\Lambda(t) x)  )}{t}>z_j,\ \forall j\in J\bigg)\; dz .\label{critiint1}
\end{eqnarray}
To find a simpler expression, we define for all $J\subset \{1,...,k\}$ and $z=(z_1,...,z_k)\in \nbR_+^{*k}$ the vector $z^J$ of which the $j$th entry $z^J_j$ is $z_j$ if $j\in J$, and some arbitrary negative value (e.g. $-1$) otherwise. With this, we can drive a more compact form of \eqref{critiint1} given by
\begin{eqnarray}
\gamma_t(x)&= &\sum_{J\subset \{1,...,k\}} (-1)^{\mbox{card}(J)+k+1} \nonumber\\
&&  \times\int_{\nbR_+^{*k}} \prod_{j=1}^k [s_j \exp(s_j z_j)] \; \Lambda(t)\; \nbP\bigg(\frac{N^o(   t - \Lambda^{-1} (\Lambda(t) x)  )}{t}>z^J \bigg)\; dz .\label{critiint2}
\end{eqnarray}
%\begin{equation}\label{critiint2}
%\gamma_t(x)= \sum_{J\subset \{1,...,k\}} (-1)^{\mbox{card}(J)+1} \int_{\nbR_+^{*k}} \prod_{j=1}^k [s_j \exp(s_j z_j)] \; t\; \nbP\bigg(\frac{N^o(tx)}{t}>z^J\bigg)\; dz,
%\end{equation}
where, for two vectors $v_1$ and $v_2$, $v_1>v_2$ means that each entry of $v_1$ is larger than the corresponding one in $v_2$.

Next, let us observe that, for a fixed $x\in (0,1)$, from (\ref{Laminv}) it follows that
$
 t - \Lambda^{-1} (\Lambda(t) x) =t - \lambda_\infty^{-1} \Lambda(t) x - \eta ( \Lambda(t) x) \Lambda(t) x =t - \lambda_\infty^{-1} \Lambda(t) x + o(t)
$
and also $\lambda_\infty^{-1} \Lambda(t) x= \lambda_\infty^{-1} [\lambda_\infty t + o(t)] x = tx + o(t)$ due to $\lim_{t\to +\infty}\Lambda(t)/t=\lambda_\infty$. Thus, one finds that
\begin{equation}\label{equiv_lambda}
 t - \Lambda^{-1} (\Lambda(t) x) \sim t(1-x),\quad t\to+\infty .
\end{equation}
Since the above result entails that $t - \Lambda^{-1} (\Lambda(t) x)\longrightarrow +\infty$ as $t\to +\infty$, from \cite[Theorems 1 and 5]{W70}, we find for $x\in(0,1)$ that
\begin{eqnarray}
[t - \Lambda^{-1} (\Lambda(t) x)] \; \nbP(N^o(t - \Lambda^{-1} (\Lambda(t) x))>0)&\longrightarrow & \beta , \label{conv_critical_tool1}\\
\nbP\left(\left. \frac{N^o( t - \Lambda^{-1} (\Lambda(t) x) )}{t - \Lambda^{-1} (\Lambda(t) x)}> z \right| N^o(t - \Lambda^{-1} (\Lambda(t) x) )>0 \right)& \longrightarrow & \exp\left( -c \max_{i=1,...,k} \frac{z_i}{v_i  \mu_i^{-1} }\right) \nonumber \\
&& =\nbP({\cal X}>z)\label{conv_critical_tool2}
%=\left( -\frac{c}{x} \max_{i=1,...,k} \frac{z^J_i}{v_i \mu_i}\right)\\
%&& =\left( -\frac{c}{x} \max_{i\in J} \frac{z_i}{v_i \mu_i}\right)=\nbP({\cal X}>z)
\end{eqnarray}
as $t\to +\infty$ and for all $z=(z_1,...,z_k)\in {\nbR_+^{*k}}$, where we recall that $\cal X$ has a distribution given by \eqref{distribution_X}. Here again, the relation '$>$' is understood entrywise. It is noted that \eqref{conv_critical_tool2} simply states that the distribution of $\frac{N^o( t - \Lambda^{-1} (\Lambda(t) x) )}{t - \Lambda^{-1} (\Lambda(t) x)}$ given that $N^o(t - \Lambda^{-1} (\Lambda(t) x) )>0$ converges to the distribution of ${\cal X}$. Also, since $z\in \nbR^k \mapsto \nbP({\cal X}>z)$ is continuous (extending the definition in \eqref{distribution_X} from $z\in \nbR_+^{*k}$ to $z\in \nbR^k$ by putting $\nbP({\cal X}>z)=1$ if $\max_{i=1,...,k}z_i \le 0$), and $\lim_{t\to +\infty}\frac{t}{t - \Lambda^{-1} (\Lambda(t) x)} = \frac{1}{1-x}$ from (\ref{equiv_lambda}), one has from Lemma \ref{lem_appendixA} (See Appendix A) for all $z$ that
\begin{multline*}
\nbP\bigg(\left. \frac{N^o(   t - \Lambda^{-1} (\Lambda(t) x)  )}{t}>z \right| N^o(   t - \Lambda^{-1} (\Lambda(t) x)  )>0 \bigg)\\
=\nbP\bigg(\left. \frac{N^o(   t - \Lambda^{-1} (\Lambda(t) x)  )}{t - \Lambda^{-1} (\Lambda(t) x)}> \frac{t}{t - \Lambda^{-1} (\Lambda(t) x)}z \right| N^o(   t - \Lambda^{-1} (\Lambda(t) x)  )>0 \bigg)\\
\longrightarrow \nbP \left( {\cal X} > \frac{1}{1-x} z\right) = \nbP((1-x){\cal X}>z),\quad t\to +\infty ,
\end{multline*}
for a fixed $x\in (0,1)$. This latter convergence along with \eqref{equiv_lambda} and \eqref{conv_critical_tool1} entails that the components of the integrand in (\ref{critiint2}) satisfies
\begin{multline}
\Lambda(t)\; \nbP\bigg(\frac{N^o(   t - \Lambda^{-1} (\Lambda(t) x)  )}{t}>z^J \bigg)\\
= \frac{\Lambda(t)}{ t - \Lambda^{-1} (\Lambda(t) x)} \nbP\bigg(\left. \frac{N^o(   t - \Lambda^{-1} (\Lambda(t) x)  )}{t}>z^J \right| N^o(   t - \Lambda^{-1} (\Lambda(t) x)  )>0 \bigg) \\
\times [t - \Lambda^{-1} (\Lambda(t) x)] \; \nbP(N^o(t - \Lambda^{-1} (\Lambda(t) x))>0)\\
% \frac{1}{x} \nbP\left(\left. \frac{N^o(tx)}{tx}> \frac{z^J}{x} \right| N^o(tx)>0 \right) \; tx\; \nbP(N^o(tx)>0)\\
\longrightarrow \frac{\lambda_\infty}{1-x} \beta\; \nbP((1-x)\; {\cal X}>z^J),\quad t\to +\infty , \quad \forall J\subset \{1,...,k\}.\nonumber
\end{multline}
It is important to note that from the convergence result in \eqref{conv_critical_tool1}, $[t - \Lambda^{-1} (\Lambda(t) x)] \; \nbP(N^o(t - \Lambda^{-1} (\Lambda(t) x))>0)$ is bounded uniformly in $t\ge 0$ and $x\in (0,1)$ by some constant. Also, $\frac{\Lambda(t)}{ t - \Lambda^{-1} (\Lambda(t) x)}$ is upper bounded in $t\ge 0$ by some constant that depends on $x$ as it is convergent towards $\frac{\lambda_\infty}{1-x}$ as $t\to +\infty$. Therefore, the following function is bounded by 
$$
\Lambda(t)\; \nbP\bigg(\frac{N^o(   t - \Lambda^{-1} (\Lambda(t) x)  )}{t}>z^J \bigg) \le  K_x, \quad \forall J\subset \{1,...,k\},\quad \forall t\ge 0,
$$
where $K_x$ is some constant independent from $t\ge 0$ and $z\in \nbR_+^{*k}$.
Since the integral $\int_{\nbR_+^{*k}} \prod_{j=1}^k [s_j \exp(s_j z_j)] dz$ is finite for  fixed $s=(s_1,...,s_k)\in (-\infty ,0)^k$, one finds by the dominated convergence theorem that the integrand in (\ref{critiint2}) satisfies
\begin{multline*}
\int_{\nbR_+^{*k}} \prod_{j=1}^k [s_j \exp(s_j z_j)] \;       
 \Lambda(t)\; \nbP\bigg(\frac{N^o(   t - \Lambda^{-1} (\Lambda(t) x)  )}{t}>z^J \bigg)
\; dz\\
\longrightarrow \int_{\nbR_+^{*k}} \prod_{j=1}^k [s_j \exp(s_j z_j)] \bigg\{\frac{\lambda_\infty}{1-x} \beta\; \nbP((1-x)\; {\cal X}>z^J)\bigg\} dz ,\quad t\to+\infty , \quad \forall J\subset \{1,...,k\}
\end{multline*}
for a fixed $x\in (0,1)$. Putting this into \eqref{critiint2} yields that (\ref{gammat}) converges to
%\begin{multline*}
$$\gamma_t(x) \longrightarrow \sum_{J\subset \{1,...,k\}} (-1)^{\mbox{card}(J)+k+1} \int_{\nbR_+^{*k}} \prod_{j=1}^k [s_j \exp(s_j z_j)] \bigg\{\frac{\lambda_\infty}{1-x} \beta\; \nbP((1-x)\; {\cal X}>z^J)\bigg\}dz,\quad t\to +\infty $$
%\end{multline*}
for a fixed $x \in (0,1)$. By applying the argument leading to the expression \eqref{critiint2} for $\gamma_t(x)$ in (\ref{gammat}), it can be shown that the right-hand side of the above convergence is $\gamma(x)$ in (\ref{critiint}). Thus, \eqref{critiint} is proved.
%And, since $\lambda(t-tx)\longrightarrow \lambda(\infty) $ as $t\to +\infty$ for each $x\in (0,1)$, one obtains that the integrand in \eqref{criti} converges towards
%%\begin{equation}\label{critiint3}
%$$\gamma_t(x)\lambda(t-tx) \longrightarrow \lambda(\infty) \gamma(x)= \frac{\lambda(\infty) \beta}{x} \nbE  [1-\exp  (<s, x\;  {\cal X}> )  ],\quad x\in (0,1).$$
%%\end{equation}

{\bf Step 3: Proof of \eqref{limit_to_prove}.} Thanks to \eqref{domination_gamma_t_x} and \eqref{critiint}, by the dominated convergence theorem, one thus deduces that \eqref{criti} converges as $t\to +\infty$ to
\[
-\int_0^1 \gamma(x) dx= -\int_0^1 \frac{\lambda_\infty \beta}{1-x} \nbE[1-\exp (<s, (1-x)\;  {\cal X}>)  ] dx, 
\]
which results in \eqref{limit_to_prove} after changing a variable $y:=-\ln (1-x)$.

{\bf Step 4: End of proof.} From (\ref{LTNtt}) with the convergence results of (\ref{criti}) towards \eqref{limit_to_prove}, one finds that
\begin{equation}\label{asyLTNtt}
\nbE\bigg[\exp \bigg(<s, \frac{N(t)}{t}>\bigg)\bigg] \longrightarrow
\exp\left( \lambda_\infty \beta \int_0^\infty \nbE[\exp (<s, e^{-y} {\cal X}>)-1  ]dy \right),\quad t\rightarrow +\infty,
\end{equation}
for $ s\in (-\infty,0]^k$, 
Since ${\cal X}=\chi v\otimes  \mu^{-1} $ with $\chi\sim {\cal E}(c)$, one computes  that $ \nbE [\exp  (<s, e^{-t} {\cal X}> )-1 ]= \nbE [\exp  (<s, v\otimes  \mu^{-1} > \chi e^{-t} )-1 ] =\frac{e^{-t} <s, v\otimes  \mu^{-1} >}{c- e^{-t} <s, v\otimes  \mu^{-1} >}$. In turn, changing of variable $z:=e^{-t} <s, v\otimes  \mu^{-1} >$ yields that the right-hand side of the above convergence is the LT equivalent to 
$\left(\frac{c}{c-<s, v\otimes  \mu^{-1} >} \right)^{\lambda_\infty \beta}$, which indeed is the LT of ${\cal Z} v\otimes  \mu^{-1} $ in \eqref{conv_distrib_critical}. This completes the proof.
%As in the proof of Theorem \ref{T1} Point (1), Campbell's formula says that the right-hand side in the above convergence is the LT of 

\section{Immigration modelled by Generalized Polya process (GPP)}\label{sec:GPP}

As discussed in Section \ref{sec:intro}, the GPP became a well-known contagion model when the transition intensity in the non-homogeneous birth process is a linear function of the current state multiplied by a function of the current time. In this section, we now assume that the arrival process $\{S(t),\ t\ge 0 \}$ is the GPP (or a positive contagion model in \cite{B70, W10}), i.e. a particular case of self exciting counting process with intensity rate $\lambda(t)$ satisfying 
 \begin{equation}\label{lam}
\lambda(t)=[aS(t^-)+b]\lambda_t,\qquad a >0,\  b> 0,
 \end{equation}
for some underlying function $t\mapsto \lambda_t>0$. When $b=1$, this arrival process was referred to as Linear Extension of the Yule Process (LEYP) by \cite{L14}. Hence, the intensity increases linearly with the number of arrivals at time $t$, which explains why such a model could be appropriate for the situations where the arriving particles representing cells infected by rapidly expanding disease contaminate other cells in an organism modelled by certain network mechanism or where the occurrence of shocks causes  outages of interconnected lines in a power system as studied in \cite{QJS16}. Let us start by establishing the LT of $N(t)$ as obtained in Lemma \ref{lemLTZt} for the NHPP immigration. 

\begin{lemm}\label{lemLTZtGPP}\normalfont
When the new particle arrives according to the GPP with the intensity rate given in (\ref{lam}), the LT of $N(t)$ in (\ref{Nt}) admits the following expression
 \begin{equation}\label{LTNt2}
  \varphi_t(s)=\bigg\{1- \int^t_0 [\varphi^o_{t-y}(s)-1] a \lambda_y  e^{a \Lambda_y} dy \bigg\}^{-b/a},
 \end{equation}
 \begin{proof}
 It is known that the marginal distribution of $S(t)$ is expressed as a negative binomial distribution with $\Lambda_t=\int_0^t \lambda_y dy$ (e.g. \cite[Theorem 1(i)]{C14}) given by 
 \[
 p_t(n):=\nbP(S(t)=n)=\frac{\Gamma(b/a+n)}{\Gamma(b/a)n!}
 (1-e^{-a \Lambda_t})^n (e^{-a \Lambda_t})^{b/a},
 \]
where $\Gamma(z)=\int^\infty_0 x^{z-1}e^{-x}dx$ for $z>0$ is the gamma function, that is a negative binominal distribution $(r,p)$ where $r=b/a$ and $p=1-e^{-a \Lambda_t}$. Its probability generating function is $P_t(z)=\sum^\infty_{n=0} z^n p_t(n)=(\frac{1-p}{1-pz})^r$ where $|z|<p^{-1}$. 

Then, from \cite[Section 3.2]{LWX14}, the LT of $N(t)$ can be expressed as a compound Negative binomial distribution as
 \begin{equation}\label{LTNt}
 \varphi_t(s)=P_t(\widetilde{f}_t(s)),
 \end{equation}
 where the LT of the secondary distribution is given by
 \begin{equation}\label{LTsec}
 \widetilde{f}_t(s)=\int^t_0 q_t(y) \varphi^o_{t-y}(s)dy
 \end{equation}
 with 
 \begin{equation}\label{qty}
 q_t(y)=\frac{a \lambda_y  e^{a \Lambda_y}}{e^{a \Lambda_t}-1},\qquad 0\leq y\leq t.
 \end{equation}
 Since $P_t(z)=(\frac{1-p}{1-pz})^r$, (\ref{LTNt}) is obtained as
 \[
  \varphi_t(s)=\bigg( \frac{e^{-a \Lambda_t}}
  {1-(1-e^{-a \Lambda_t}) \widetilde{f}_t(s)}\bigg)^{b/a}=\Big[ e^{a  \Lambda_t}-(e^{a\Lambda_t}-1)\widetilde{f}_t(s)\Big]^{-b/a}.
 \]
 But $\int^t_0 a\lambda_y e^{a  \Lambda_y} dy=e^{a \Lambda_t}-1$, so that one finds from (\ref{LTsec}) and (\ref{qty}) that
 \begin{align*}
e^{a \Lambda_t}-(e^{a \Lambda_t}-1)\widetilde{f}_t(s)&=  \int^t_0 a \lambda_y  e^{a \Lambda_y} dy + 1 - \int^t_0 a \lambda_y  e^{a \Lambda_y} \varphi^o_{t-y} (s) dy\\
&=1+ \int^t_0 [1-\varphi^o_{t-y}(s)] a \lambda_y  e^{a  \Lambda_t} dy.
 \end{align*}
 That is,
 \begin{equation}\label{LTNt1}
  \varphi_t(s)=\bigg\{1+ \int^t_0 [1-\varphi^o_{t-y}(s)] a \lambda_y  e^{a \Lambda_y} dy \bigg\}^{-b/a},
 \end{equation}
or equivalently (\ref{LTNt2}).
 \end{proof}
\end{lemm}

Although the result in Lemma \ref{lemLTZtGPP} holds for a general function $t\mapsto \lambda_t$ in (\ref{lam}), we shall focus on the case when $\lambda_t=\lambda >0$ is constant in the following. In this case, $\{S(t),\ t\ge 0 \}$ is called a contagious Poisson process \cite{A80}, and (\ref{LTNt1}) is simplified as
 \begin{equation}\label{scGPP}
   \varphi_t(s)=\bigg\{1+ \int^t_0 [1-\varphi^o_{t-y}(s)] a \lambda  e^{a \lambda y} dy \bigg\}^{-b/a}.
 \end{equation}
In the case of a constant baseline intensity $\lambda_t=\lambda$, taking the expectation on both sides of \eqref{lam} yields $\nbE[\lambda(t)]=a\lambda \nbE[S(t^-)] + \lambda b$. Since $\nbE[S(t^-)]=\nbE[S(t)]$ and $\{S(t)-\int_0^t \lambda(s)ds,\ t\ge 0\}$ is a martingale, we arrive at $\nbE[\lambda(t)]=a\lambda \int_0^t \nbE[\lambda(s)]ds + \lambda b$ for all $t\ge 0$, from which the expected intensity has the closed form
\begin{equation}\label{intensity_expected}
\nbE[\lambda(t)]= b \lambda e^{a\lambda t},\quad t\ge 0 .
\end{equation}
We note that there is some resemblance between this exponential expression in \eqref{intensity_expected} in the GPP case and the exponential asymptotic form $\lambda(t)\sim \lambda_\infty e^{\delta t}$ of the (deterministic) intensity
appeared in Theorems \ref{T1} and \ref{T1bis} in the NHPP case. However, due to the randomness feature in time of the intensity in this case, it is expected to observe different limiting behaviors for the branching process $N(t)$ with the GPP immigration.
%In the case of the GPP immigration, due to the form of the intensity in \eqref{lam} leading to a rapid growth of the immigration rate as $t$ tends to infinity, it is anticipated that asymptotic behaviour of $N(t)$ and the notion of subcritical, supercritical and critical cases are different from the ones in the NHPP given in Section \ref{sec:NHPP}.
More precisely, in the following it is shown that the distributional behaviour changes depending on whether the largest eigenvalue $\rho$ of $A$ is less than, larger than, or equal to $a\lambda$. The main result of this section is given in the following theorem.
%For the GPP immigration, we have the following asymptotic results for $N(t)$, according to whether the largest eigenvalue $\rho$ of $A$ is less than, larger than, or equal to $a\lambda$. These will be referred as subcritical, supercritical and critical cases. 
 
 \begin{theorem}\label{T3}\normalfont
One has the following convergences in distribution:  \\
\noindent (1) When $\rho>a\lambda$, %(supercritical case), 
\begin{equation}\label{GPP_sup}
e^{-\rho t} N(t) \stackrel{\cal D}{\longrightarrow} \Z_T\ v,\qquad t\rightarrow +\infty,
\end{equation}
where $T\sim \Gamma(b/a,1)$ and $\{ \Z_t,\ t\ge 0\}$ is an independent L\'evy process with characteristic exponent $\psi(x):=\int_\nbR \left( 1- \exp \left[ -x z\right]\right)\Pi (dz)$, $x\ge 0$. Here, $\Pi(.)$ is defined by
\begin{equation}
%\xi&:=& \int_0^\infty \int_0^\infty  w^{a\lambda/\rho }\nbu_{[w\ge z]} \nbP(W\in dw)\ \frac{a\lambda}{\rho} z^{-a\lambda/\rho-1} dz\nonumber\\
% &=& \int_0^\infty \nbE \left( W^{a\lambda/\rho }\nbu_{[W\ge z]}\right) \ \frac{a\lambda}{\rho} z^{-a\lambda/\rho-1} dz, \label{def_xi_superc}\\
%f_W(z) & :=& \frac{1}{\xi} \nbE \left( W^{a\lambda/\rho }\nbu_{[W\ge z]}\right) \ \frac{a\lambda}{\rho} z^{-a\lambda/\rho-1}. \label{def_fW_superc}
\Pi (dz)  :=  \nbE \left[W^{a\lambda/\rho }\nbu_{[W\ge z]}\right] \ \frac{a\lambda}{\rho} z^{-a\lambda/\rho-1}\nbu_{[0<z<+\infty]}\ dz . \label{def_Pi_superc}
\end{equation}
\noindent (2) 
When $\rho<a\lambda$,% (subcritical case),  
\begin{equation}\label{GPP_sub}
e^{-a\lambda t} N(t) \stackrel{\cal D}{\longrightarrow}  {\cal{Z}}\ \gamma,
\qquad t\rightarrow +\infty,
\end{equation}
where $\cal{Z}$ is a rv distributed as $\Gamma(b/a ,1)$ and $\gamma$ is the vector defined by 
\begin{equation}\label{def_gamma_GPP}
\gamma:=a\lambda (a\lambda I-A)^{-1} \mathbf{n}_0 %\int^{+\infty}_0 a\lambda e^{(A-a\lambda I)y}dy.\mathbf{n}_0.
\end{equation}
\noindent (3) When $\rho=a\lambda$,% (critical case), 
\begin{equation}\label{GPP_cri}
\frac{N(t)}{t}e^{-a\lambda t}\stackrel{\cal D}{\longrightarrow} {\cal Z} v,\qquad t\rightarrow +\infty,
\end{equation}
where ${\cal Z}$ is a rv distributed as $\Gamma (b/a, \nbE[W]a \lambda )$.
\end{theorem}
\begin{remark}\normalfont
In the case $\rho>a\lambda$ we may note that, since $\Pi(.)$ defined by \eqref{def_Pi_superc} has support on $(0,+\infty)$ and verifies $\int_{(0,+\infty)} \min(1,z)\Pi (dz) <+\infty$ (precisely because of the condition $ \rho>a\lambda$), the underlying L\'evy process $\{ \Z_t,\ t\ge 0\}$ appearing in \eqref{GPP_sup} belongs to the class of subordinators according to \cite[Lemma 2.14, p.55]{K06}.
\end{remark}
\begin{remark}\normalfont
As shown in \eqref{intensity_expected}, the expected intensity in the present GPP case has an exponential form, so that it is natural to compare the limiting convergence results in Theorem \ref{T3} to those in Theorems \ref{T1} and \ref{T1bis} in the NHPP case with asymptotic intensity $\lambda(t)\sim \lambda_\infty e^{\delta t}$ for $\delta>0$. The limits in those theorems are different, however there exists an analogy between $a\lambda$ and $\delta$: Table \ref{table:dir} summarizes the different {\it directions of the supports} of the limiting distributions obtained in \eqref{subcritical_gamma_positive}, \eqref{supercritical_gamma_equal} and \eqref{convergence_supercritical} for $\rho<\delta$, $\rho=\delta$ and $\rho>\delta$ in the NHPP case, to those in \eqref{GPP_sub}, \eqref{GPP_cri} and \eqref{GPP_sup} for $\rho<a\lambda$, $\rho=a\lambda$ and $\rho>a\lambda$ in the GPP case. Each value in Table \ref{table:dir}, that all belong to $\nbR_+^k$, roughly shows the position in which the renormalized processes $N(t)$ is located asymptotically in the corresponding case.
\begin{table}[h]
\begin{tabular}{|c|c|c|c|}
\hline
 & $\rho<\xi$ & $\rho=\xi$ & $\rho>\xi$\\
 \hline
 NHPP: $\xi=\delta$&  $\displaystyle (\xi I -A)^{-1}\mathbf{n}_0$ & $u$ & $v$\\
 \hline
 GPP: $\xi=a\lambda$ & $\displaystyle (\xi I -A)^{-1}\mathbf{n}_0$ & $v$ & $v$\\
 \hline
\end{tabular}
\caption{Direction of limiting distribution}
	\label{table:dir}
\end{table}
Interestingly, the directions are the same except for $\rho=\delta$ in the NHPP case and $\rho=a\lambda$ in the GPP, which are respectively given by the vectors $u$ and $v$.
\end{remark}
The proofs of each case in Theorem \ref{T3} are provided in the following Section \ref{subsec:sc_gpp}, Section \ref{sec:superc_gpp}, and Section \ref{sec:crit_gpp} respectively.

\subsection{Proof of Theorem \ref{T3} in the case $\rho> a \lambda$}\label{sec:superc_gpp}
In \eqref{scGPP}, with a choice of the renormalization function $g(t)=e^{\rho t}$ we get
\begin{equation}\label{supercGPPgt}
 \varphi_t (s/g(t))=\varphi_t (s e^{-\rho t})=\bigg\{1+ \int^t_0 [1-\varphi^o_{t-y}(se^{-\rho t})] a \lambda  e^{a \lambda y} dy \bigg\}^{-b/a}
 \end{equation}
for all $ s \in (-\infty,0]^k$. The proof is divided into two steps as follows.

{\bf Step 1: Studying the convergence of $\varphi_t (s e^{-\rho t})$ as $t\to +\infty$.} It is convenient to introduce the function
 \begin{multline}\label{xit_superc}
\Xi_{t,s}:= \int^t_0 [1-\varphi^o_{t-y}(se^{-\rho t})] a \lambda  e^{a \lambda y} dy \\
 = \int^\infty_0 \nbu_{[0< y< t]} \nbE \left[ 1- \exp \left( <s, N^o(t-y)/e^{\rho(t-y)}> e^{-\rho y}\right)\right] a \lambda  e^{a \lambda y} dy,
 \end{multline}
so that $\varphi_t (s e^{-\rho t})=\{1+\Xi_{t,s}\}^{-b/a}$, $t\ge 0$. Thus, studying the limit of $\varphi_t (s e^{-\rho t})$ as $t\to+\infty$ essentially requires finding $\lim_{t\to +\infty} \Xi_{t,s}$, which will be completed by the dominated convergence theorem. First note that for all $y\in (0,+\infty)$ one has that $N^o(t-y)/e^{\rho(t-y)}\longrightarrow Wv$, $t\to \infty$, a.s. from Lemma \ref{LTW}. Since $s$ has non positive entries, one has that $0 \le 1- \exp \left[ <s, N^o(t-y)/e^{\rho(t-y)}> e^{-\rho y}\right]\le 1$ for a fixed $y\in (0,+\infty)$ and thus it implies that
\begin{equation}\label{superc_conv1}
\nbE \left[ 1- \exp \left( <s, N^o(t-y)/e^{\rho(t-y)}> e^{-\rho y}\right)\right]\longrightarrow\nbE \left[ 1- \exp \left( <s,v> We^{-\rho y}\right)\right],~~t\to+\infty,%\quad \forall y\in (0,1)
\end{equation}
by the dominated convergence theorem. Also, using again the inequality $1-e^{-u}\le u$ for $ u\ge 0$, the integrand in (\ref{xit_superc}) is upper bounded as
\begin{align*}
0  &\le \nbu_{[0< y< t]}\nbE \left[ 1- \exp \left( <s, N^o(t-y)/e^{\rho(t-y)}> e^{-\rho y}\right)\right] a \lambda  e^{a \lambda y}\\
&\le   \nbu_{[0< y< t]}\nbE \left[ -<s, N^o(t-y)/e^{\rho(t-y)}> e^{-\rho y}\right] a \lambda  e^{a \lambda y}\\
&=a\lambda  \nbu_{[0< y< t]}\nbE \left[ -<s, N^o(t-y)/e^{\rho(t-y)}> \right] e^{(a \lambda -\rho)y}.
\end{align*}
By the similar martingale argument applied to the one leading to \eqref{ineq_critical} for example, one can show that $\nbu_{[0< y< t]}\nbE \left[ -<s, N^o(t-y)/e^{\rho(t-y)}>\right]$ is upper bounded by some constant say $K$ which is independent of $t$ and $y$. That is, 
\begin{equation}\label{superc_conv2}
0 \le \nbu_{[0< y< t]}\nbE \left[ 1- \exp \left( <s, N^o(t-y)/e^{\rho(t-y)}> e^{-\rho y}\right)\right] a \lambda  e^{a \lambda y}\le a\lambda K e^{(a \lambda -\rho)y}
\end{equation}
which is integrable over $y\in (0,+\infty)$ when $\rho> a\lambda$. Hence, thanks to \eqref{superc_conv1} and \eqref{superc_conv2} we arrive by the dominated convergence theorem at
\begin{equation}\label{Xiinfty}
\Xi_{t,s}\longrightarrow \Xi_{\infty,s}:=\int_0^\infty \nbE \left[ 1- \exp \left( <s,v> We^{-\rho y}\right)\right] a \lambda  e^{a \lambda y} dy,\quad t\to +\infty ,
\end{equation}
so that the renormalized LT in (\ref{supercGPPgt}) converges as
\begin{equation}\label{superc_conv3}
\varphi_t (s e^{-\rho t}) \longrightarrow \tilde{\varphi}(s):= \{1+ \Xi_{\infty,s}  \}^{-b/a} ,\quad t\to+ \infty .
\end{equation}

{\bf Step 2: Identifying the LT $\tilde{\varphi}(s)$.} In order to interpret \eqref{superc_conv3} as the convergence towards some known distribution, we use the following elementary Lemma (its proof is given in Appendix B):
\begin{lemm}\label{Lemma_superc}\normalfont
Let $\{ \Z_t,\ t\ge 0\}$ be a L\'evy process with characteristic exponent $\psi(x)$ such that $\nbE[e^{-x\Z_t}] =e^{-t\psi(x)}$ for $ x\ge 0$, and let $T$ be a rv distributed as $\Gamma(\zeta ,1)$, independent from $\{ \Z_t,\ t\ge 0\}$. Then the LT of $\Z_T$ is given by
\begin{equation}\label{LT_Lemma_superc}
\nbE[e^{-x \Z_T}]= \{1+  \psi(x) \}^{-\zeta},\quad x\ge 0 .
\end{equation}
%Let $\{ Z_t,\ t\ge 0\}$ be a compound Poisson process with intensity $\xi>0$ and positive jumps with cdf $F(.)$, and let $T$ be a rv distributed as $\Gamma(\zeta ,1)$, independent from $\{ Z_t,\ t\ge 0\}$. Then the LT of $Z_T$ is given by
%\begin{equation}\label{LT_Lemma_superc}
%\nbE(e^{-x Z_T})= \bigg\{1+ \xi \int_0^\infty (1-e^{-xz}) F(dz) \bigg\}^{-\zeta},\quad \forall x\ge 0 .
%\end{equation}
\end{lemm}
The aim is now to write $\tilde{\varphi}(s)$ in (\ref{superc_conv3}) in the form of \eqref{LT_Lemma_superc}. We first write $\Xi_{\infty,s}$ in (\ref{Xiinfty}) as
$$
\Xi_{\infty,s}=\int_0^\infty \int_0^\infty \left( 1- \exp \left[ <s,v> w e^{-\rho y}\right]\right) a \lambda  e^{a \lambda y} dy \ \nbP(W\in dw).
$$
Performing a change of variable $z:=we^{-\rho y}$ (i.e. $y=-\frac{1}{\rho} \ln \frac{z}{w}$) within the integral in $y$, it may be expressed as
\begin{eqnarray*}
\Xi_{\infty,s} & = & \int_0^\infty  \int_0^w \left( 1- \exp \left[ <s,v> z\right]\right) \frac{a\lambda}{\rho} \left( \frac{z}{w} \right)^{-a\lambda/\rho } \frac{dz}{z}  \ \nbP(W\in dw)\\
&=& \int_0^\infty \left( 1- \exp \left[ <s,v> z\right]\right) \bigg\{\int_0^\infty  w^{a\lambda/\rho }\nbu_{[w\ge z]} \nbP(W\in dw)\bigg\} \frac{a\lambda}{\rho} z^{-a\lambda/\rho-1} dz\\
%&=& \xi \int_0^\infty \left( 1- \exp \left[ <s,v> z\right]\right) f_W(z) dz
&=&  \int_0^\infty \left( 1- \exp \left[ <s,v> z\right]\right)\ \nbE \left[ W^{a\lambda/\rho }\nbu_{[W\ge z]}\right] \ \frac{a\lambda}{\rho} z^{-a\lambda/\rho-1} dz \\
&=& \int_\nbR \left( 1- \exp \left[ <s,v> z\right]\right)\Pi (dz),
\end{eqnarray*}
where the measure $\Pi (dz)$ on $(0,+\infty)$ is defined as \eqref{def_Pi_superc}. Finally, we get the following expression for (\ref{superc_conv3}):
$$
\tilde{\varphi}(s) = \big\{1+  \psi(<s,v>) \big\}^{-b/a},\quad s\in (-\infty , 0]^k ,
$$
so that one deduces from Lemma \ref{Lemma_superc} the convergence result in (\ref{GPP_sup}). 
%$$
%e^{-\rho t} N(t) \stackrel{\cal D}{\longrightarrow} \Z_T\ v,\quad t\to \infty ,
%$$
%where $T\sim \Gamma(b/a,1)$ and $\{ \Z_t,\ t\ge 0\}$ is a L\'evy process with characteristic exponent $\psi(x):=\int_\nbR \left( 1- \exp \left[ -x z\right]\right)\Pi (dz)$, $x\ge 0$, where $\Pi(.)$ is defined by \eqref{def_Pi_superc}, and $T\sim \Gamma(b/a,1)$. Note that, since $\Pi(.)$ has support on $(0,+\infty)$ and verifies $\int_{(0,+\infty)} \min(1,z)\Pi (dz) <+\infty$ (because of the supercritical condition $ \rho>a\lambda$), $\{ \Z_t,\ t\ge 0\}$ belongs to the class of subordinators according to \cite[Lemma 2.14 p.55]{K06}.

 \subsection{ Proof of Theorem \ref{T3} in the case $\rho<a\lambda$}\label{subsec:sc_gpp}
 
 After changing a variable $y:=t-y$, (\ref{scGPP}) is rewritten as
\[
\varphi_t(s)=\bigg\{1+ \int^t_0 [1-\varphi^o_{y}(s)] a \lambda  e^{a \lambda (t-y)} dy \bigg\}^{-b/a},\qquad t\geq 0,~~s \in (-\infty, 0]^k.
\]
 Let us consider the renormalizing function $g(t)=e^{a\lambda t}$, so that
 \begin{equation}\label{scGPPgt}
 \varphi_t (s/g(t))=\varphi_t (s e^{-a\lambda t})=\bigg\{1+ \int^t_0 [1-\varphi^o_{y}(se^{-a\lambda t})] a \lambda  e^{a \lambda (t-y)} dy \bigg\}^{-b/a}.
 \end{equation}
 In the following, the limit of the integral on the right-hand side of (\ref{scGPPgt}) is studied in the subcritical case. First, similar to (\ref{xit_superc}), let
 \begin{equation}\label{xit}
\Xi_{t,s}:= \int^t_0 [1-\varphi^o_{y}(se^{-a\lambda t})] a \lambda  e^{a \lambda (t-y)} dy.
 \end{equation}
 To apply the dominated convergence theorem, let us define
 \begin{align}\label{xits}
 \Xi_{t,s,y} &:=\nbu_{[0< y< t]} [1-\varphi^o_{y}(se^{-a\lambda t})] a \lambda  e^{a \lambda (t-y)}
  =  \nbu_{[0< y< t]}\nbE[1-e^{<s, N^o(y)> e^{-a \lambda t}}] a \lambda  e^{a \lambda (t-y)}.
  %\nonumber\\
% &\leq \nbu_{[0< y< t]}  \nbE[|1-e^{<s, N^o(y)> e^{-a \lambda t}}|]  a \lambda  e^{a \lambda (t-y)}
 \end{align}
Since
\begin{equation}\label{upper}
  1-e^{<s, N^o(y)> e^{-a \lambda t}}  
  \leq 
  %|<s, N^o(y)>| e^{-a \lambda t} 
  -<s,N^o(y)> e^{-a \lambda t},
\end{equation}
where the last equality is due to the negative entries in $s$, (\ref{xits}) is bounded by
\begin{align}\label{dom}
\Xi_{t,s,y} & \leq -\nbu_{[0< y< t]} \nbE[<s,N^o(y)>] a\lambda e ^{-a \lambda y} 
\leq - \nbE[<s,N^o(y)>] a\lambda e ^{-a \lambda y} \nonumber\\
&= -<s, \nbE[N^o(y)]> a\lambda e ^{-a \lambda y}  :=\Xi^\ast_{s,y}.
\end{align}
We recall from \cite[p.202]{A72} that the mean matrix of the multitype process $N^o(t)$ is expressed as $\nbE[N^o(y)]=e^{Ay}\mathbf{n}_0$ where the matrix $A$ is defined in (\ref{matrixA}) and $\mathbf{n}_0=(1,0,...,0)$. For the case $\rho<a\lambda $, the integral 
$\int^{+\infty}_0 e^{(A-a\lambda I)y} dy$ is convergent because all eigenvalues of the matrix $A-a\lambda I$ have negative real part in the case of $\rho <a\lambda$. In turn, one concludes that $\int^{+\infty}_0 \Xi^\ast_{s,y}dy$  converges. Therefore, for a fixed $y\in (0,+\infty)$ one finds 
\begin{equation}\label{pwlimit}
\nbE[1-e^{<s, N^o(y)> e^{-a \lambda t}}]   e^{a \lambda t} \longrightarrow -\nbE [<s, N^o(y)>],\qquad t\rightarrow +\infty
\end{equation}
by the dominated convergence theorem. Indeed, from (\ref{upper}) 
$| 1-e^{<s, N^o(y)> e^{-a \lambda t}}|e^{a \lambda t}$ is upper bounded by $-<s,N^o(y)>$ which has a finite expectation. Finally, because of the bound for the integrand $\Xi_{t,s,y}$ obtained in (\ref{dom}) and the pointwise limit in (\ref{pwlimit}), one deduces that (\ref{xit}) converges to
\begin{align*}
\Xi_{t,s} &\longrightarrow - \int^{+\infty}_0 \nbE(<s,N^o(y)>) a\lambda e^{-a \lambda y} dy = - \int^{+\infty}_0 <s, e^{Ay} \mathbf{n}_0> a\lambda e^{-a \lambda y} dy\\
&~~~= - <s, \int^{+\infty}_0 a\lambda e^{(A-a\lambda I)y}dy.\mathbf{n}_0>=<s,a\lambda (a\lambda I-A)^{-1} \mathbf{n}_0>,
\end{align*}
as $t\rightarrow +\infty$. Consequently, it follows that (\ref{scGPPgt}) converges to
\[
\varphi_t(se^{-a\lambda t}) \longrightarrow  \bigg\{1-<s, a\lambda (a\lambda I-A)^{-1} \mathbf{n}_0> \bigg\}^{-b/a},\qquad t\rightarrow +\infty,
\]
for all $s\in (-\infty, 0]^k$, which entails (\ref{GPP_sub}) with the vector $\gamma$ defined as (\ref{def_gamma_GPP}).
% \[
%e^{-a\lambda t} N(t) \stackrel{\cal D}{\longrightarrow}  {\cal{Z}}\ \gamma,
%\qquad t\rightarrow +\infty,
%\]
%where $\cal{Z}$ is a rv distributed as $\Gamma(b/a ,1)$ and $\gamma$ is the vector \marge{to do later: point out that the distributions in convergences in the theorems have support in different directions according to whether we are sub, super or critical} defined by 
%\[
%\gamma:= \int^{+\infty}_0 a\lambda e^{(A-a\lambda I)y}dy.\mathbf{n}_0
%\]

\subsection{Proof of Theorem \ref{T3} in the case $\rho= a\lambda$}\label{sec:crit_gpp}
We consider here the renormalizing function $g(t):= t e^{\rho t}= t e^{a\lambda t}$. As in \eqref{supercGPPgt} and \eqref{xit_superc}, after changing a variable $y:=y/t$ we have for all $ s \in (-\infty,0]^k$,
\begin{eqnarray}
 \varphi_t (s/g(t))&=&\varphi_t (s e^{-a\lambda t}/t)=\bigg\{1+ \int^t_0 [1-\varphi^o_{t-y}(se^{-a\lambda t}/t)] a \lambda  e^{a \lambda y} dy \bigg\}^{-b/a}\nonumber\\
 &=& \bigg\{1+ \int^t_0 \left(1- \nbE\left[ \exp(<s, N^o(t-y)> e^{-a\lambda t}/t)
 \right] \right) a \lambda  e^{a \lambda y} dy \bigg\}^{-b/a}\nonumber\\
 &=& \bigg\{1+ \int^1_0 t\left(1-\nbE\left[ \exp(<s, N^o(t(1-y))> e^{-a\lambda t}/t)
 \right] \right) a \lambda  e^{a \lambda t y} dy \bigg\}^{-b/a}\nonumber\\
 &=& \{1+ \Xi_{t,s} \}^{-b/a},\label{critGPPgt}
\end{eqnarray}
where $\Xi_{t,s}$ is now defined by
\begin{eqnarray}
\Xi_{t,s} &:=& \int^1_0 t\left(1-\nbE\left[ \exp(<s, N^o(t(1-y))> e^{-a\lambda t}/t\right]\right) a \lambda  e^{a \lambda t y} dy\nonumber\\
&=& \int^1_0 t\left(1-\nbE\left[ \exp(<s, Wv> e^{-a\lambda ty }/t)\right]\right) a \lambda  e^{a \lambda t y} dy\nonumber\\
&~~+& \int^1_0 t\left(\nbE\left[ \exp(<s, Wv> e^{-a\lambda ty }/t)\right] - \nbE\left[ \exp(<s, N^o(t(1-y))> e^{-a\lambda t}/t)\right] \right) a \lambda  e^{a \lambda t y} dy\nonumber\\
&:=& \Xi^1_{t,s} + \Xi^2_{t,s}. \label{crit_conv1}
\end{eqnarray}
In the following we shall determine the limits of $\Xi^1_{t,s}$ and 
$\Xi^2_{t,s}$ separately as $t\to + \infty$. For notational convenience, let $\Xi^2_{t,s}:=\int_0^1 \Upsilon_s^2(t,y) dy$ where
\begin{equation}\label{up2}
\Upsilon^2_s(t,y):=  t\left(\nbE\left[ \exp(<s, Wv> e^{-a\lambda ty }/t)\right] - \nbE\left[ \exp(<s, N^o(t(1-y))> e^{-a\lambda t}/t)\right] \right) a \lambda  e^{a \lambda t y}.
\end{equation}

{\bf Step 1: Studying the convergence of $\Xi^1_{t,s}$ as $t\to + \infty$.} It is readily obtainable that using the inequality $0\le 1-e^x\le -x$ for $x\le 0$, one has for all $t\ge 0$ and $y\in (0,1)$ that
\begin{align*}
0 &\le t\left[1- \exp(<s, Wv> e^{-a\lambda ty }/t)\right] a \lambda  e^{a \lambda t y}\\
 &\le - t   <s, Wv> (e^{-a\lambda ty }/t) a \lambda  e^{a \lambda t y}=-   <s, Wv>   a \lambda,% \label{crit_conv2}
\end{align*}
which is integrable, so that for a fixed $y\in (0,1)$ one has by the dominated convergence theorem that $ t\left(1-\nbE\left[ \exp(<s, Wv> e^{-a\lambda ty }/t)\right]\right) a \lambda  e^{a \lambda t y} \longrightarrow -\nbE\left[ <s, Wv> \right] a \lambda$ as $t\to +\infty$. Likewise:
$$
0\le t\left(1-\nbE\left[ \exp(<s, Wv> e^{-a\lambda ty }/t)\right]\right) a \lambda  e^{a \lambda t y}
\le  -\nbE\left[ <s, Wv> \right] a \lambda,  
$$
a constant, so that by the dominated convergence theorem one deduces that
\begin{equation}\label{step1}
\lim_{t\to +\infty}   \Xi^1_{t,s} =-\nbE\left[<s, Wv> \right] a \lambda= - <s, \nbE[W] a \lambda v>.
\end{equation}

{\bf Step 2: Dominating $\Upsilon^2_s(t,y)$.} In order to study $\lim_{t\to +\infty}\Xi^2_{t,s}$, we again use the dominated convergence theorem. First, it can be shown that $|\Upsilon^2_s(t,y)|$ in (\ref{up2}) is upper bounded by some constant as:
\begin{eqnarray}
|\Upsilon^2_s(t,y)| &\le & t\nbE[ | <s, Wv> e^{-a\lambda ty }/t - <s, N^o(t(1-y))> e^{-a\lambda t}/t | ] a \lambda  e^{a \lambda t y} \nonumber\\
&= & a\lambda \nbE[ (| <s, Wv>   - <s, N^o(t(1-y))> e^{-a\lambda t(1-y)} | ]\label{step2}\\
&\le & a\lambda \nbE [ | <s, Wv>| ]  + a\lambda\nbE [ | <s, N^o(t(1-y))> e^{-a\lambda t(1-y)}| ]\nonumber\\
&=& - a\lambda \nbE [  <s, Wv> ]  - a\lambda \nbE [  <s, N^o(t(1-y))> e^{-a\lambda t(1-y)}],\nonumber
\end{eqnarray}
where the first inequality is obtained from (\ref{A}) and the last equality holds because $W$ and $N^o(t(1-y))$ are non negative or have non negative entries and $s$ has negative entries.  Using again the constant $\kappa$ satisfying \eqref{sj} and the martingale argument, one thus obtains together with the above result that
\begin{equation*}\label{crit_conv3}
|\Upsilon^2_s(t,y)| \le - a\lambda \nbE[ <s, Wv> ]  - a\lambda \kappa  <u, {\bf n}_0> ,\quad \forall t\ge 0,\ \forall y\in (0,1).
\end{equation*}

{\bf Step 3: Pointwise convergence of $\Upsilon^2_s(t,y)$ towards $0$ as $t\to +\infty$.} Let $y\in (0,1)$ be fixed. Since $\nbR^k$ can be decomposed as the direct sum of $\nbR u$ and $\left( \nbR v \right)^{\bot} $ (the orthogonal vector space of $\nbR v$ for the euclidian inner product), there exists some (unique) $\alpha\in\nbR$ and $s_0\in \left( \nbR v \right)^{\bot} $ such that $ s=\alpha u + s_0$. Since $<s_0,v>=0$, it follows that (\ref{step2}) is expressed as
\begin{eqnarray*}
&&|\Upsilon^2_s(t,y)| \le  a\lambda \nbE[|<s, Wv>   - <s, N^o(t(1-y))> e^{-a\lambda t(1-y)} | )\\
&=& a\lambda \nbE( | \alpha < u, Wv>   - \alpha <u, N^o(t(1-y))> e^{-a\lambda t(1-y)} - <s_0, N^o(t(1-y))> e^{-a\lambda t(1-y)}|).
%&\le & a\lambda |\alpha| 
\end{eqnarray*}
Since $<u,Wv>=W <u,v>=W.1=W$, using the triangle inequality followed by Cauchy Schwarz inequality yields 
\begin{multline}\label{crit_conv4}
|\Upsilon^2_s(t,y)| \le a\lambda |\alpha| \Big\{ \nbE\big[ \big|  W  - <u, N^o(t(1-y))> e^{-a\lambda t(1-y)} \big|^2 \big]\Big\}^{1/2}\\
+ a\lambda  \Big\{ \nbE\big[\big|<s_0, N^o(t(1-y))> e^{-a\lambda t(1-y)}\big|^2\big] \Big\}^{1/2}.
\end{multline}
Then it will be shown that both terms on the right-hand side of \eqref{crit_conv4} tend to $0$ as $t\to +\infty$. The reason why $s$ is decomposed along $\nbR u$ and $\left( \nbR v \right)^{\bot} $ is that the first term is linked to the martingale $\{<u,N^o(t)e^{-\rho t}>,\ t\ge 0\}=\{<u,N^o(t)e^{-a\lambda t}>,\ t\ge 0\}$, whereas  in the second term the behaviour of $\{<s_0,N^o(t)e^{-\rho t}>,\ t\ge 0\}$ may be controlled precisely because $s_0\in \left( \nbR v \right)^{\bot}$ thanks to the estimates given in \cite{A69} . Indeed, one has from \cite[(iii) p.204]{A72} that $\left(\nbE[|| N^o(t)||^2 e^{-2\rho t}]\right)_{t\ge 0}$ is uniformly upper bounded with $\rho=a\lambda$ here. Since $\nbE[ | <u, N^o(t)> e^{-a\lambda t} |^2]$ is upper bounded by $\nbE[|| N^o(t)||^2 e^{-2\rho t}]$ up to a constant for all $t\ge 0$, one deduces that the martingale $\{<u,N^o(t)e^{-a\lambda t}>,\ t\ge 0\}$ is uniformly square integrable, hence converges in mean square towards $W$ as $t\to +\infty$; and in turn, the first term on the right-hand side of \eqref{crit_conv4} converges to $0$ 	as $t\to +\infty$. And, from \cite[Proposition 3]{A69} together with $<s_0,v>=0$, there exists some real number $a(s_0)<\rho=a\lambda$ as well as an integer $\gamma(s_0)$ (both depending on $s_0$, see their precise definitions in \cite[(9a) and (9b)]{A69}) such that one of the three following situations occurs:
%\begin{equation}\label{crit_conv5}
$$\nbE\Big[\big|<s_0, N^o(t)>\big|^2\Big] =\left\{
\begin{array}{cl}
O(e^{2a(s_0)t} t^{2\gamma(s_0)})& \mbox{if }2a(s_0)>\rho=a\lambda,\\
O(e^{2a(s_0)t} t^{2\gamma(s_0)+1})& \mbox{if }2a(s_0)=\rho=a\lambda,\\
O(e^{\rho t})=O(e^{a\lambda t}) & \mbox{if }2a(s_0)<\rho=a\lambda.
\end{array}
\right.
$$
%\end{equation} 
Here the above three cases are corresponding to \cite[a), b) and c) of Proposition 3]{A69} respectively. In all cases, since $a(s_0)$ verifies $a(s_0)<\rho=a\lambda$, one checks easily that $\nbE[|<s_0, N^o(t)>|^2] e^{-2\rho t}=\nbE [|<s_0, N^o(t)>|^2] e^{-2a\lambda t}$ tends to $0$ as $t\to +\infty$. Hence the second term in the right-hand side of \eqref{crit_conv4} tends to $0$ as $t\to +\infty$ (for a fixed $y\in (0,1)$). Combining all the above results, we thus prove that both terms on the right-hand side of \eqref{crit_conv4} converge to $0$. Therefore, it is concluded that (\ref{up2}) goes to zero as $t\to+\infty$ for all $y\in (0,1)$.

{\bf Step 4: Convergence of $\Xi^2_{t,s}$ and conclusion.} Step 2 and Step 3 imply by the dominated convergence theorem that $\lim_{t\to +\infty}\Xi^2_{t,s}=0$. Then together with (\ref{step1}), from (\ref{crit_conv1}) it follows that \eqref{critGPPgt} converges to
%\begin{equation}\label{crit_conv6}
$$\varphi_t (s e^{-a\lambda t}/t)\longrightarrow \big\{1 - <s, \nbE[W]a \lambda v>   \big\}^{-b/a}, \quad t\to+\infty ,$$
%\end{equation}
so that we proved (\ref{GPP_cri}).
% that
%$$
%\frac{N(t)}{t}e^{-a\lambda t}\stackrel{\cal D}{\longrightarrow} {\cal Z} v,\quad t\to +\infty ,
%$$
%where ${\cal Z} \sim \Gamma (b/a, \nbE(W)a \lambda )$.

 \section{Transient expectation when $k=2$}\label{sec:app}
We shall hereafter consider two-type branching processes (i.e. $k=2$) to study transient expectation of the number of particles at time $t$. Assume that the lifetime of type $j$ particles for $j=1,2$ is exponentially distributed as ${\cal E}(\mu_i)$. The branching mechanism is given by the following generating functions (see Definition \eqref{def_gen_fc})
$$
h_1(z_1,z_2)= p_1(0,0)+ p_1(0,1) z_2,\quad
h_2(z_1,z_2)= p_2(0,0)+ p_2(1,0) z_1,\quad (z_1,z_2)\in [0,1]^2,
$$
where probabilities $p_{12}:=p_1(0,1)$ and $p_{21}:=p_2(1,0)$ in $(0,1]$ satisfy $ p_{12}p_{21}<1$, which means that type 1 particle (resp. $2$) produces a type 2 (resp. $1$) particle with probability $p_{12}$ (resp. $p_{21}$), or else dies with probability $p_1(0,0)=1-p_{12}$ (resp. $p_2(0,0)=1-p_{21}$). Finally, we denote by $t\ge 0 \mapsto m(t)=\nbE[S(t)]$ the renewal function associated to the immigration process $\{ S(t),\ t\ge 0\}$.
\begin{theorem}\normalfont
At time $t$, the transient expectation $\nbE[N_1(t)]$ for type 1 particle is given by
\begin{multline}\label{N_1_transient6}
\nbE [N_1(t)]= p_{12}\int_0^t \int_0^{t-y} \left[\Psi(t-y)-\Psi(t-y-z)\right]\mu_1 e^{-\mu_1 z}dz \; dm(y)\\
+ (1-p_{12})\int_0^t m(t-s)\Psi(ds)  ,\quad t\ge 0 ,
\end{multline}
where $\Psi(ds)$ is given by 
\begin{eqnarray}
\Psi(ds)&=& \delta_0(ds) + \mu_1 \mu_2 p_{12}p_{21}\left[ \frac{1}{\zeta_1 (\zeta_2-\zeta_1)}e^{\zeta_1 s} + \frac{1}{\zeta_2 (\zeta_1-\zeta_2)}e^{\zeta_2 s}\right]\; ds,\ s\ge 0 , \label{expression_psi_transient}\\
\mbox{with}\quad \zeta_1 &:=& \frac{1}{2}\left[ -(\mu_1+\mu_2) + \sqrt{(\mu_1-\mu_2)^2 + 4\mu_1\mu_2  p_{12}p_{21}}\right],\label{expre_zeta1}\\
\zeta_2 &:=& \frac{1}{2}\left[ -(\mu_1+\mu_2) - \sqrt{(\mu_1-\mu_2)^2 + 4\mu_1\mu_2  p_{12}p_{21}}\right].\label{expre_zeta2}
\end{eqnarray}
\end{theorem}
Similar analysis is available to obtain a transient expression for $\nbE[N_2(t)]$ for type 2 particles. Note that the expression \eqref{N_1_transient6} depends on the renewal function $m(t)$, which is explicitly available in many processes. For example, $m(t)=\int_0^t\lambda(s)\; ds$ when the immigration process is an NHPP with intensity $\lambda(\cdot)$ whereas $m(t)=(\frac{b}{a})\frac{1- e^{-a\Lambda_t}}{e^{-a\Lambda_t}}$ when the immigration process is GPP with parameters $(a,b,t\mapsto \lambda_t)$. In addition to these two processes considered in this paper, we remark that \eqref{N_1_transient6} for the transient first moment is also available for other non Poisson arrival processes where their renewal functions are known. Typical examples include the case when $\{ S(t),\ t\ge 0\}$ is a fractional Poisson process with parameter $\beta\in (0,1)$ (in which case $m(t)=Ct^\beta$ for some constant $C>0$, see \cite[Expression (26)]{L03}), or when the interarrival times $T_i-T_{i-1}$, $i\ge 1$, follow matrix exponential distributions (in which case $m(t)$ is explicit and given by \cite[Theorem 3.1]{AB96}). 
\begin{proof}
The key idea is to consider the successive passage times from type 2 to type 1 of the $i$th particle arriving at $T_i$, $i\in \nbN^*$ which is type 1. The type of particles is changing between 1 and 2 while it remains in the same type during an exponentially distributed lifetime as long as it is alive (i.e. it has not left the system). Let us introduce the sequence $(V_i^{(r)})_{r\in \nbN}$ representing the succesive time instants of this particle (arriving at time $T_i$) changing back to type 1 after being type 2. In other words, this $i$th particle becomes type 1 again  at the times $T_i + V_i^{(1)}$, $T_i + V_i^{(2)}$, etc. if it has not left the system in between. Then the sequence is expressed as $V_i^{(r)} - V_i^{(r-1)}=W_i^{(r)}$ from $r\ge 1$ with $V_i^{(0)}=0$ where $\{W_i^{(r)},\ r\in \nbN^* \}$ is an iid sequence of defective random variables, with
$$
\nbP(W_i^{(r)}=+\infty)=1-p_{12}p_{21},\quad {\cal D}( W_i^{(r)}| W_i^{(r)}<\infty)={\cal E}(\mu_1)\star {\cal E}(\mu_2),
$$
where $\star$ stands for the convolution operator.
Here, the event $[W_i^{(r)}=+\infty]$ corresponds to the case when the $i$th particle dies (i.e. exits the system) on its $r$th sojourn with type 1 or type 2. It is convenient in the following to write $W_i^{(r)}=Y_{1,i}^{(r)}+Y_{2,i}^{(r)}$ where $Y_{j,i}^{(r)}$ represents the $r$th sojourn time of type $j$ particle for $j=1,2$ and $Y_{1,i}^{(r)}$ independent from $Y_{2,i}^{(r)}$. The distributions are given by ${\cal D} ( Y_{j,i}^{(r)}| Y_{j,i}^{(r)}<\infty )=  {\cal E}(\mu_j)$, $\nbP(Y_{1,i}^{(r)}=\infty)=1-p_{12}$ and $\nbP(Y_{2,i}^{(r)}=\infty)=1-p_{21}$. See Figure \ref{graph_transient} for an illustration.
\begin{figure}[!h]%
\centering
\includegraphics[scale=0.8]{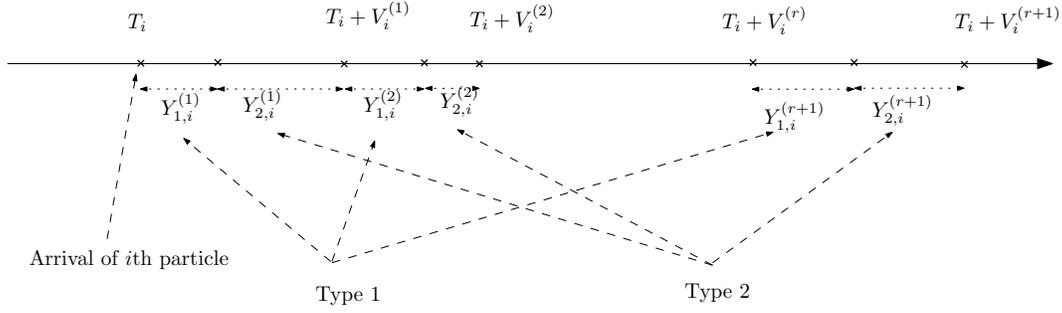}%
\caption{\label{graph_transient} Evolution of $i$th particle.}
\end{figure}
Then, $N_1(t)$ has the following expression
\begin{equation}\label{N_1_transient1}
N_1(t)=\sum_{i=1}^\infty \sum_{r=0}^\infty \nbu_{[T_i + V_i^{(r)}\le t < T_i + V_i^{(r)} + Y_{1,i}^{(r+1)}]},
\end{equation}
%which counts the number of particles present in system \#$1$ at time $t$
as $[T_i + V_i^{(r)}\le t < T_i + V_i^{(r)} + Y_{1,i}^{(r+1)}]$ corresponds to the event that type 1 particle arrived at time $T_i$ is again type 1 at time $t$ after its $r$th return time. Taking the expectation in \eqref{N_1_transient1} and interchanging the order of summation yields 
\begin{eqnarray}
\nbE [N_1(t)]&=& \sum_{r=0}^\infty B_r (p_{12}p_{21})^r,\label{N_1_transient2}\\
B_r &:=& \sum_{i=1}^\infty\nbP(  T_i + V_i^{(r)}\le t < T_i + V_i^{(r)} + Y_{1,i}^{(r+1)}|  V_i^{(r)}<\infty ),\label{N_1_transient3}
\end{eqnarray}
with $[V_i^{(r)}<\infty]=[W_m^{(r)}<\infty,\ m=1,...,r] = [Y_{1,i}^{(m)}<\infty, Y_{2,i}^{(m)}<\infty,\ m=1,...,r]$. Further conditioning on $Y_{1,i}^{(r+1)}$ either $\infty$ or $<\infty$ in \eqref{N_1_transient3} leads to
\begin{multline}\label{N_1_transient4}
B_r= p_{12}\sum_{i=1}^\infty\nbP( T_i + V_i^{(r)}\le t < T_i + V_i^{(r)} + Y_{1,i}^{(r+1)}|  V_i^{(r)}<\infty , Y_{1,i}^{(r+1)} <\infty)\\
+ (1-p_{12})\sum_{i=1}^\infty\nbP( T_i + V_i^{(r)}\le t |  V_i^{(r)}<\infty , Y_{1,i}^{(r+1)} =\infty):=B^1_r+B^2_r,\quad r\ge 0.
\end{multline}
In the following, explicit expressions for $B^1_r$ and $B^2_r$ are derived. Let us denote by $G^{(r)}(\cdot)$ to be the cumulative distribution function (cdf) of $( V_i^{(r)}| V_i^{(r)}<\infty)$ which has the same distribution as $( W_i^{(j)}| W_i^{(j)}<\infty)^{\star (r)}$, in other words, the $r$th convolution of the sum of two exponentials with mean $\mu_1$ and mean $\mu_2$ and also denote $G^{(0)}(ds)=\delta_0(ds)$. Since $Y_{1,i}^{(r+1)}$ is independent from $V_i^{(r)}$, $B_r^2$ in (\ref{N_1_transient4}) admits the expression
\begin{equation}\label{N_1_transient5}
B^2_r= (1-p_{12})\int_0^t m(t-s) G^{(r)}(ds) ,\quad r\ge 0 .
\end{equation}
Also, $T_i$ is independent from $Y_{1,i}^{(r+1)}$ and $V_i^{(r)}$ and $Y_{1,i}^{(r+1)}$ and $V_i^{(r)}$ are identically distributed as $Y_{1,1}^{(r+1)}$ and $V_1^{(r)}$ respectively. Then, we find that $B_r^1$ in (\ref{N_1_transient4}) is expressed as
\begin{multline*}
B^1_r = p_{12}\int_0^t \nbP ( V_1^{(r)} \le t-y < V_1^{(r)} + Y_{1,1}^{(r+1)}| V_1^{(r)}<\infty , Y_{1,1}^{(r+1)} <\infty)\; dm(y)\\
= p_{12}\int_0^t \int_0^{t-y}  \left[G^{(r)}(t-y)-G^{(r)}(t-y-z)\right] \mu_1 e^{-\mu_1 z}dz \; dm(y),\quad r\ge 0.
\end{multline*}
Then using the above expression together with \eqref{N_1_transient5} for \eqref{N_1_transient4}, from \eqref{N_1_transient2} it follows that $\nbE[N_1(t)]$ is given by \eqref{N_1_transient6} where $\Psi(ds)$ is a distribution defined by
\begin{equation}\label{N_1_transient7}
\Psi(ds)=\sum_{r=0}^\infty (p_{12}p_{21})^r G^{(r)}(ds).
\end{equation}
Since $G^{(r)}(ds)$ is the distribution of the sum of two independent Erlang distributions with respective parameters $(r,\mu_1)$ and $(r,\mu_2)$, its LT is given by
$$\widehat{G^{(r)}}(x)=\int_0^\infty e^{-xs} G^{(r)}(ds)=
\left(\frac{\mu_1}{\mu_1+x} \frac{\mu_2}{\mu_2+x}\right)^r,\quad r\ge 0,
  \quad x\ge 0 ,$$
so that, taking the LT on both sides of \eqref{N_1_transient7}, one obtains
\begin{multline}\label{N_1_transient8}
\widehat{\Psi}(x)=\sum_{r=0}^\infty (p_{12}p_{21})^r\widehat{G^{(r)}}(x)=\frac{1}{1- p_{12}p_{21}\frac{\mu_1}{\mu_1+x} \frac{\mu_2}{\mu_2+x}}=1+\frac{\mu_1 \mu_2 p_{12}p_{21}}{x^2 + (\mu_1+\mu_2)x + \mu_1\mu_2(1- p_{12}p_{21} )}\\
=1+\frac{\mu_1 \mu_2 p_{12}p_{21}}{(x-\zeta_1)(x-\zeta_2)}=1+ \mu_1 \mu_2 p_{12}p_{21}\left[ \frac{1}{(\zeta_1-\zeta_2)(x-\zeta_1)} + \frac{1}{(\zeta_2-\zeta_1)(x-\zeta_2)}\right],
\end{multline}
where $\zeta_1$ and $\zeta_2$ are defined by \eqref{expre_zeta1} and \eqref{expre_zeta2}. Inverting \eqref{N_1_transient8} then yields \eqref{expression_psi_transient}.
%again by independence of $Y_{1,i}^{(r+1)}$ is independent from $V_i^{(r)}$, and since ${\cal D} \left(\left.Y_{1,i}^{(r+1)}\right| Y_{1,i}^{(r+1)}<\infty \right)=  {\cal E}(\mu_1)$:
\end{proof}

\bibliographystyle{alpha}

%\section*{Appendix}

%\subsection*{A. }

\section*{Appendix A: A uniform estimate}
Polya's theorem states that if some real valued rv $X_t$ tends in distribution towards $X\in \nbR$ then the corresponding cdf converges uniformly, i.e. $\sup_{x\in \nbR} |\nbP(X_t \le x) - \nbP (X\le x)| \longrightarrow 0$ as $t\to +\infty$, provided that $x\mapsto \nbP (X\le x)$ is continuous. This in particular implies that if $\{x_t,\ t\ge 0\}$ is a sequence of real numbers such that $\lim_{t\to \infty} x_t=x$ then $\lim_{t\to \infty}\nbP(X_t \le x_t) = \nbP (X\le x)$. Although it seems that a multidimensional version of this fact is less known, we present here a proof of it for the sake of completeness.
\begin{lemm}\label{lem_appendixA}\normalfont
Let $\{X_t,\ t\ge 0\}$ be a sequence of random variables with values in $\nbR^k$ converging in distribution towards $X\in \nbR^k$, such that $x\in \nbR^k \mapsto \nbP(X \le x)$ is continuous. Then one has for all $x\in \nbR^k$ that
$$
\lim_{t\to +\infty}\nbP(X_t > x_t) = \nbP (X> x)
$$
where $\lim_{t\to \infty} x_t=x$, $x_t$ lying in $\nbR^k$, and  '$\le$' and '$>$' are understood componentwise.
\end{lemm}
\begin{proof}
Let us recall that the symmetric difference of two sets $A$ and $B$ is defined by $A\Delta B:= [A\setminus B]\cup [B\setminus A]= [A\cap \bar{B}]\cup [B\cap \bar{A}]$ where $\bar{A}$ is the complimentary of the set $A$, which satisfies $|\nbP(A)-\nbP(B)|\le \nbP(A\Delta B)$. Thus, writing $x=(x^1,...,x^k)$, $x_t=(x^1_t,...,x^k_t)$, $X=(X^1,...,X^k)$ and $X_t=(X^1_t,...,X^k_t)$, we have:
\begin{eqnarray}
|\nbP(X_t > x_t) - \nbP(X_t > x)|&\le & \nbP ([X_t > x_t]\Delta [X_t > x])\nonumber\\
& = & \nbP ([X_t > x_t]\setminus [X_t > x]) + \nbP ([X_t > x]\setminus [X_t > x_t])\nonumber\\
&=& \nbP\left( \left\{ \cap_{j=1}^k [X_t^j > x_t^j]\right\} \cap \left\{ \cup_{j=1}^k [X_t^j \le x^j] \right\}\right)\nonumber\\
&& + \nbP\left( \left\{ \cap_{j=1}^k [X_t^j > x^j]\right\} \cap \left\{ \cup_{j=1}^k [X_t^j \le x^j_t] \right\}\right).\label{AppendixA1}
\end{eqnarray}
Note that the following inclusion of event holds: $\{ \cap_{j=1}^k [X_t^j > x_t^j]\} \cap \{ \cup_{j=1}^k [X_t^j \le x^j] \}\subset \cup_{j=1}^k [x_t^j< X_t^j \le x^j]$, with the convention that $[x_t^j< X_t^j \le x^j]=\varnothing$ if $x_t^j\ge x^j$; hence one deduces that $\nbP( \{ \cap_{j=1}^k [X_t^j > x_t^j]\} \cap \{ \cup_{j=1}^k [X_t^j \le x^j] \})\le \sum_{j=1}^k \nbP (x_t^j< X_t^j \le x^j)$. Since convergence in distribution of $X_t$ towards $X$ obviously means convergence in distribution of each entry $X_t^j$ towards $X^j$, Polya's theorem implies that $\nbP (x_t^j< X_t^j \le x^j)= \nbP(X_t\le x^j) - \nbP(X_t\le x^j_t)\longrightarrow 0$ as $t\to +\infty$ for all $j=1,...k$, as indeed continuity of $x\in \nbR^k \mapsto \nbP(X \le x)$ implies continuity of $ x\in\nbR\mapsto \nbP(X^j \le x)$ for all $j=1,...,k$. Thus, one gets that $\lim_{t\to +\infty}\nbP ( \{ \cap_{j=1}^k [X_t^j > x_t^j]\} \cap \{ \cup_{j=1}^k [X_t^j \le x^j] \})=0$. Similarly one has that $\lim_{t\to +\infty} \nbP ( \{ \cap_{j=1}^k [X_t^j > x^j]\} \cap \{ \cup_{j=1}^k [X_t^j \le x^j_t] \})=0$, so the proof is completed from \eqref{AppendixA1} along with the fact that $\lim_{t\to +\infty} \nbP(X_t > x) = \nbP (X> x)$.
\end{proof}

\section*{Appendix B: Proof of Lemma \ref{Lemma_superc}}
The LT of $T\sim \Gamma(\zeta ,1)$ is given by $\nbE(e^{-xT})=(1+x)^{-\zeta}$ for $x\ge 0$. %, and it is standard that the LT of $Z_t$ (for fixed $t\ge 0$) is given thanks to Campbell's formula by $\nbE(e^{-xZ_t})=\exp\left(-t\xi \int_0^\infty (1-e^{-xz}) F(dz)\right)$, see e.g. \cite[Formula (2.9) in Theorem 2.7 p.41]{K06}.
By the independence assumption, one then gets
\begin{eqnarray*}
\nbE[e^{-x\Z_T}]&=& \int_0^\infty \nbE[e^{-x\Z_t}] \nbP(T\in dt)= \int_0^\infty e^{-t\psi(x)}\nbP(T\in dt)\\
&=& \{1+  \psi(x) \}^{-\zeta},
\end{eqnarray*}
which completes the proof.

\end{document}